\documentclass[11pt]{article}
\usepackage[utf8]{inputenc}
\usepackage[a4paper,top=2.5cm,bottom=2cm,left=2cm,right=2cm,marginparwidth=1.5cm]{geometry}
\usepackage{amsmath}
\usepackage{amssymb}
\numberwithin{equation}{section}
\usepackage[english]{babel}
\usepackage{amsthm}
\usepackage{bbm}
\usepackage{amsfonts}
\usepackage{hyperref}
\hypersetup{
    colorlinks=true,
    linkcolor=black,
    filecolor=magenta,      
    urlcolor=black,
    pdftitle={A mathematical design strategy for highly dispersive resonator systems},
    pdfpagemode=FullScreen,
    }
\usepackage{graphicx, tikz}
\usepackage{float}
\usepackage{xcolor}

\usepackage[format=plain,
            font=it]{caption}

\newtheorem{theorem}{Theorem}[section]
\newtheorem{corollary}{Corollary}[theorem]
\newtheorem{lemma}[theorem]{Lemma}
\newtheorem{proposition}[theorem]{Proposition}
\newtheorem*{remark}{Remark}
\newtheorem{definition}[theorem]{Definition}

\newcommand{\upd}{\mathrm{d}}  

\def\a{\alpha }     \def\b{\beta  }         \def\g{\gamma }
     \def\d{\delta}          \def\D{\Delta }
    \def\ve{\varepsilon}    
\def\h{\eta }               
\def\k{\kappa }          \def\l{\lambda }    
    \def\m{\mu}             \def\n{\nu}         
    \def\r{\rho}                 
     \def\t{\tau }

       \def\w{\omega }         \def\W{\Omega }

\title{A mathematical design strategy for highly dispersive resonator systems\thanks{The work of KA was supported by the project ETH-34 20-2 and the work of BD was supported by EC H2020 FETOpen project BOHEME under grant agreement No.~863179.}}
\author{Konstantinos Alexopoulos\thanks{Department of Mathematics, ETH Zurich, R\"amistrasse 101, CH-8092 Zurich, Switzerland, konstantinos.alexopoulos@sam.math.ethz.ch.} \and Bryn Davies\thanks{Department of Mathematics, Imperial College London, 180 Queen's Gate, London SW7 2AZ, United Kingdom, bryn.davies@imperial.ac.uk.}}
\date{}
\begin{document}

\maketitle

\begin{abstract}
    Designing devices composed of many small resonators is a challenging problem that can easily incur significant computational cost. Can asymptotic techniques be used to overcome this often limiting factor? Integral methods and asymptotic techniques have been used to derive concise characterisations for scattering by resonators, but can these be generalised to systems of many dispersive resonators whose material parameters have highly non-linear frequency dependence? In this paper, we study halide perovskite resonators as a demonstrative example. We extend previous work to show how a finite number of coupled resonators can be modelled concisely in the limit of small radius. We also show how these results can be used as the basis for an inverse design strategy, to design resonator systems that resonate at specific frequencies.
\end{abstract}


\tableofcontents

\section{Introduction}

When multiple resonators are allowed to influence one another coupling interactions take place, which can often be complex and difficult to model. Understanding how these interactions depend on the shapes, sizes and positions of resonators has allowed scientists and engineers to design devices with exotic and remarkable properties. Some notable examples include effectively negative material parameters \cite{pendry2000negative, veselago1968electrodynamics}, cloaking devices \cite{milton2006cloaking, craster2012acoustic} and bio-inspired structural colouration \cite{sun2013structural, zhao2012bio}. The word \emph{metamaterial} is a broad term that is often used as a catchall term to encompass materials whose emergent properties arise due to geometry and structure (as opposed to purely from chemistry) \cite{kadic20193dmetamaterials}.

For designing complex devices, it is valuable to be able to model systems of coupled resonators without the need for expensive numerical simulations (\emph{e.g.} with commercial finite element packages). For this reason, there has been significant mathematical interest in developing concise models for coupled resonator systems. A prominent field in this direction is multiple scattering theory \cite{tsang2004scattering}. These techniques are particularly effective for modelling either small (point) scatterers or systems whose geometry admits explicit representations (\emph{e.g.} cylinders or spheres) \cite{TB, gower2019multiple}. To describe resonators with general, possibly complex, shapes, integral methods can be used \cite{AFKRYZ}. On top of this, asymptotic techniques have helped provide concise characterisations of complex problems. Homogenisation can be used characterise the effective properties of materials with periodic \cite{BBF, GZ}, quasi-periodic \cite{bouchitte2010homogenization} or random \cite{foldy1945multiple} micro-structures. Local properties can also be deduced through asymptotic approaches. For example, asymptotic expansions can be computed when resonators are very small or have highly contrasting material parameters \cite{ADFMS, ADH}.

Extending existing asymptotic and integral methods to models of dispersive resonators, with physically realistic material parameters, has proved to be a challenging problem. Some recent progress has been made for the well-known Drude model \cite{baldassari2021modal} and for halide perovskites \cite{AD}. In these cases, resonant frequencies of the coupled resonator system cannot be found by solving a simple eigenvalue problem, as the associated eigenvalue problem inherits the non-linearity of the permittivity relation. 

In this work, we will focus on halide perovskites, as a demonstrative example of the asymptotic and integral techniques we will exploit. Halide perovskites are materials which are increasingly being used in optical devices. Their underly chemisty consists of octohedral-shaped crystalline lattices containing atoms of heavier halides, such as chlorine, bromine and iodine \cite{AM}. When used in microscopic devices, their high absorption coefficient helps absorb the complete visible spectrum. This, combined with the fact that they are cheap and easy to manufacture, means they are playing a prominent role in the production of electromagnetic devices \cite{GPLLZZSQKJ, JKM,MFTHBPZK,S,WKYZZPLBHL}.

In this paper, we will use integral methods to study a broad class of geometries of halide perovskite resonators. This extends the theory developed in \cite{AD} for one and two resonators to the case of three or more halide perovskite nano-particles. In section~\ref{sec:asymptotics}, we will present the integral formulation of the resonance problem that we are studying. We will use asymptotic techniques to show how this system can be approximated in the case that the resonators are small. In section~\ref{sec:frequencies}, we will show how these results can be used to find the resonant frequencies of a coupled system of circular halide perovskite resonators and present numerical visualisations. Our results will be for a two-dimensional differential system, however we will show (in the appendix) how these results can easily be modified to three dimensions.

In the final part of this paper, in section~\ref{sec:inverse}, we will use our asymptotic results to treat an inverse design problem. In particular, given three wavelengths of visible light, we will show that a system of three identical circular halide perovskite resonators can be chosen to resonate at those wavelengths and present an efficient strategy for deriving the appropriate geometry. This problem is inspired by the sensitivity of retinal receptor cells to three colours of light (red, blue and green). This shows that, with the help of our mathematical insight, it is possible to add customisable colour perception to bioinspired artificial eyes \cite{GPLLZZSQKJ, lee2018bioinspired}.

\section{Asymptotic analysis} \label{sec:asymptotics}

\subsection{Problem setting}

Let us consider $N\in\mathbb{N}$ halide perovskite resonators $D_1,D_2,\dots,D_N$ occupying a bounded domain $\Omega \subset \mathbb{R}^d$, for $d\in \{2,3\}$. We assume that the resonators have permittivity given by
\begin{align}\label{prmtvt}
    \ve(\w,k) = \ve_0 + \frac{\a}{\b - \w^2 + \h k^2 - i \g \w},
\end{align}
where $\a,\b,\g,\h$ are positive constants. This is motivated by the formula for the permittivity of halide perovskites reported in \cite{MFTHBPZK}. The non-linear dependence on both the frequency $\w$ and the wavenumber $k$ are responsible for the complex, dispersive behaviour of the material. We assume that the particles are non-magnetic, so that the magnetic permeability $\m_0$ is constant on all of $\mathbb{R}^d$. 

We consider the Helmholtz equation as a model for the propagation of time-harmonic waves with frequency $\w$. This is a reasonable model for the scattering of transverse magnetic polarised light (see \emph{e.g.} \cite[Remark~2.1]{moiola2019acoustic} for a discussion). The wavenumber in the background $\mathbb{R}^d \setminus \overline{\W}$ is given by $k_0:=\w\ve_0\m_0$ and we will use $k$ to denote the wavenumber within $\W$. Let us note here that, from now on, we will suppress the dependence of $k_0$ on $\w$ for brevity.  We, then, consider the following Helmholtz model for light propagation:
\begin{align}\label{Helmholtz problem}
    \begin{cases}
    \D u + \w^2 \ve(\w,k)\m_0 u = 0 \ \ \ &\text{ in } \W, \\
    \D u + k_0^2 u = 0 &\text{ in } \mathbb{R}^d\setminus\overline{\W}, \\
    u|_+ - u|_-=0 &\text{ on } \partial\W, \\
    \frac{\partial u}{\partial \n}|_+ - \frac{\partial u}{\partial \n}|_- = 0 &\text{ on } \partial\W, \\
    u(x) - u_{in}(x) &\text{ satisfies the outgoing radiation condition as } |x|\to \infty,\\
    \end{cases}
\end{align}
where $u_{in}$ is the incident wave, assumed to satisfy
\begin{align*}
    (\D + k_0^2)u_{in}=0 \ \ \text{ in } \mathbb{R}^d,
\end{align*}
and the appropriate outgoing radiation condition is the Sommerfeld radiation condition, which requires that
\begin{equation}\label{Sommerfeld}
    \lim_{|x|\to\infty} |x|^{\frac{d-1}{2}} \left( \frac{\partial}{\partial |x|} - i k_0 \right)\Big(u(x) - u_{in}(x)\Big)=0.
\end{equation}
In particular, we are interested in the case of small resonators. Thus, we will assume that there exists some fixed domain $D$, which the the union of $N$ disjoint subsets $D = D_1 \cup D_2 \cup \dots\cup D_N$, such that $\W$ is given by
\begin{equation}
    \W = \d D + z,
\end{equation}
for some position $z\in\mathbb{R}^d$ and characteristic size $0<\d\ll1$. Then, making a change of variables, the Helmholtz problem (\ref{Helmholtz problem}) becomes
\begin{align}\label{2.4}
    \begin{cases}
    \D u + \d^2 \w^2 \ve(\w,k)\m_0 u = 0 \ \ \ &\text{ in } D, \\
    \D u + \d^2 k_0^2 u = 0 &\text{ in } \mathbb{R}^d\setminus\overline{D}, \\
    \end{cases}
\end{align}
along with the same transmission conditions on $\partial D$ and far-field behaviour. We are interested in the \emph{subwavelength} behaviour of the system, which occurs when $\d \ll k_0^{-1}$. We will study this by performing asymptotics in the regime that the frequency $\w$ is fixed while the size $\d\to0$. We will characterise solutions to (\ref{Helmholtz problem}) in terms of the system's resonant frequencies. For a given wavenumber $k$, we define $\w=\w(k)$ to be a \emph{resonant frequency} if it is such that there exists a non-trivial solution $u$ to (\ref{Helmholtz problem}) in the case that $u_{in}=0$.

\subsection{Integral formulation}

Let $G(x,k)$ be the outgoing Helmholtz Green's function in $\mathbb{R}^d$, defined as the unique solution to $(\D + k^2) G(x,k) = \d_0(x)$ in $\mathbb{R}^d$, along with the outgoing radiation condition (\ref{Sommerfeld}). It is well known that $G$ is given by
\begin{align}
    G(x,k)=
    \begin{cases}
    -\frac{i}{4}H^{(1)}_0(k|x|), \ \ \ &d=2,\\
    -\frac{e^{ik|x|}}{4\pi|x|}, &d=3,\\
    \end{cases}
\end{align}
where $H_0^{(1)}$ is the Hankel function of first kind and order zero. Then, from \emph{e.g.} \cite{AD}, we have the following result, which gives an integral representation of the scattering problem.

\begin{theorem}[Lippmann-Schwinger integral representation formula] \label{Integral representation}
The solution to the Helmholtz problem (\ref{Helmholtz problem}) is given by
\begin{align}\label{Lipmann-Schwinger}
    u(x)-u_{in}(x) = -\d^2 \w^2 \xi(\w,k) \int_D G(x-y,\d k_0) u(y) \upd y, \ \ x\in\mathbb{R}^d,
\end{align}
where the function $\xi:\mathbb{C}\to\mathbb{C}$ describes the permittivity contrast between $D$ and the background and is given by
$$
\xi(\w,k) = \m_0(\ve(\w,k) - \ve_0).
$$
\end{theorem}

Since the domains $D_1,\dots,D_N$ are disjoint, the field $u-u_{in}$ scattered by the $N$ particles can be written as
\begin{align}\label{repping}
    (u-u_{in})(x) = - \d^2 \w^2 \xi(\w,k) \sum_{i=1}^N \int_{D_i} G(x-y,\d k_0) u(y), \ \text{ for } x \in \mathbb{R}^d.
\end{align}

We are interested in understanding how the formula (\ref{Lipmann-Schwinger}) behaves in the case that $\d$ is small. For this, the asymptotic expansions of the Green's function will be of great help. Although, we have to distinguish the cases of two and three dimensions, since these expansions differ in each case. We will work on the two-dimensional setting as the asymptotic expansions are more complicated. The same method can be used in three-dimensions, although the analysis is slightly easier. We present some of the key details in Appendix~\ref{app:3d}.

\subsection{Two-dimensional analysis}

Let us assume that we work in dimension $d = 2$ and let us consider $N$ halide perovskite resonators $D_1, D_2,\dots , D_N$, made from the same material. We define the operators $K^{\d k_0}_{D_i}$ and $R^{\d k_0}_{D_i D_j}$, for $i,j=1,2,\dots,N$, $i\ne j$, as follows.

\begin{definition} \label{def:KR}
We define the integral operators $K^{\d k_0}_{D_i}$ and $R^{\d k_0}_{D_i D_j}$, for $i,j=1,2,...,N$, by
\begin{align*}
    K^{\d k_0}_{D_i}: u\big|_{D_i} \in L^2(D_i) \longmapsto - \int_{D_i} G(x-y,\d k_0) u(y) \upd y \Big|_{D_i} \in L^2(D_i)
\end{align*}
and
\begin{align*}
    R^{\d k_0}_{D_i D_j}: u\big|_{D_i} \in L^2(D_i) \longmapsto - \int_{D_i} G(x-y,\d k_0) u(y) \upd y \Big|_{D_j} \in L^2(D_j).
\end{align*}
\end{definition}
We continue by recalling from \cite{AD} some results concerning the asymptotic behaviour of these integral operators.
\begin{definition}\label{MN}
We define the integral operators $M^{\d k_0}_{D_i}$ and $N^{\d k_0}_{D_i D_j}$ for $i,j=1,2,\dots,N$, $i\ne j$, as
\begin{align*}
    M^{\d k_0}_{D_i} := \hat{K}^{\d k_0}_{D_i} + K^{(0)}_{D_i} + (\d k_0)^2 \log(\d k_0 \hat{\g}) K^{(1)}_{D_i},
\end{align*}
and
\begin{align*}
    N^{\d k_0}_{D_i D_j} := \hat{K}^{\d k_0}_{D_i D_j} + R^{(0)}_{D_i D_j} + (\d k_0)^2 \log(\d k_0 \hat{\g}) R^{(1)}_{D_i D_j},
\end{align*}
where
\begin{align*}
    K^{(0)}_{D_i}: u\Big|_{D_i} \in L^2(D_i) &\longmapsto \int_{D_i} G(x-y,0)u(y)\upd y \Big|_{D_i} \in L^{2}(D_i),\\
    \hat{K}^{\d k_0}_{D_i}: u\Big|_{D_i} \in L^2(D_i) &\longmapsto - \frac{1}{2\pi} \log(\hat{\g} \d k_0) \int_{D_i} u(y) \upd y \Big|_{D_i}\in L^{2}(D_i), \\
    K^{(1)}_{D_i}: u\Big|_{D_i} \in L^2(D_i) &\longmapsto \int_{D_i} \frac{\partial}{\partial k} G(x-y,k)\Big|_{k=0} u(y) \upd y \Big|_{D_i} \in L^2(D_i),
\end{align*}
and
\begin{align*}
    R^{(0)}_{D_i D_j}: u\Big|_{D_i} \in L^2(D_i) &\longmapsto \int_{D_i} G(x-y,0)u(y)\upd y \Big|_{D_j} \in L^{2}(D_j),\\
    \hat{K}^{\d k_0}_{D_i D_j}: u\Big|_{D_i} \in L^2(D_i) &\longmapsto - \frac{1}{2\pi} \log(\hat{\g} \d k_0) \int_{D_i} u(y) \upd y\Big|_{D_j}\in L^{2}(D_j), \\
    R^{(1)}_{D_i D_j}: u\Big|_{D_i} \in L^2(D_i) &\longmapsto \int_{D_i} \frac{\partial}{\partial k} G(x-y,k)\Big|_{k=0} u(y) \upd y \Big|_{D_j} \in L^2(D_j).
\end{align*}
\end{definition}

\begin{proposition}
For the integral operators $K^{\d k_0}_{D_i}$ and $R^{\d k_0}_{D_i D_j}$, we can write
\begin{align} \label{M}
    K^{\d k_0}_{D_i} = M^{\d k_0}_{D_i} + O\Big(\d^4 \log(\d)\Big), \quad
    \text{and} \quad
    R^{\d k_0}_{D_i D_j} = N^{\d k_0}_{D_i D_j}+ O\Big(\d^4 \log(\d)\Big),
\end{align}
as $\delta\to0$ and with $k_0$ fixed.
\end{proposition}
Then, the resonance problem is to find $\w\in\mathbb{C}$, such that there exists $(u_1,u_2,\dots,u_N) \in L^2(D_1) \times L^2(D_2) \times \dots \times L^2(D_N)$, $u_i \ne 0$, for $i=1,\dots,N,$ such that
\begin{align}\label{2dscat}
    \begin{pmatrix}
    1-\d^2\w^2\xi(\w,k)M^{\d k_0}_{D_1} & -\d^2\w^2\xi(\w,k) N^{\d k_0}_{D_2 D_1} & \dots & -\d^2\w^2\xi(\w,k) N^{\d k_0}_{D_N D_1} \\
    -\d^2\w^2\xi(\w,k) N^{\d k_0}_{D_1 D_2} & 1-\d^2\w^2\xi(\w,k)M^{\d k_0}_{D_2} & \dots & -\d^2\w^2\xi(\w,k) N^{\d k_0}_{D_N D_2}\\
    \vdots & \vdots & \ddots & \vdots \\
    -\d^2\w^2\xi(\w,k) N^{\d k_0}_{D_1 D_N} & -\d^2\w^2\xi(\w,k) N^{\d k_0}_{D_2 D_N} & \dots & 1-\d^2\w^2\xi(\w,k)M^{\d k_0}_{D_N}
    \end{pmatrix}
    \begin{pmatrix}
    u_1 \\ u_2 \\ \vdots \\ u_N
    \end{pmatrix}
    =
    \begin{pmatrix}
    0 \\ 0 \\ \vdots \\ 0
    \end{pmatrix},
\end{align}

To ease the notation in what follows, let us define a modified version of the modulo function. This is modified to always return strictly positive values (this is important it will be used for matrix indices later). In particular, it is chosen so that $N \lfloor N \rfloor = N$ for any $N\in\mathbb{N}$.

\begin{definition}
Given $N\in\mathbb{N}$, we denote by $\lfloor N \rfloor:\mathbb{N}\to\{1,2,\dots,N\}$ a modified version of the modulo function, \emph{i.e.} the remainder of euclidean division by $N$. In particular, for all $M \in \mathbb{N}$, there exists unique $\t\in\mathbb{Z}^{\geq0}$ and $r\in\mathbb{N}$ with $0<r\leq N$, such that
$$
M = \t \cdot N + r.
$$
Then, we define $M \lfloor N \rfloor$ to be
$$
M \lfloor N \rfloor := r.
$$
\end{definition}

We now wish to make an additional assumption on the dimensions of the nano-particles. This will allow us to prove an approximation for the values of the modes $u_i$ on each particle. The assumption is one of \emph{diluteness}, in the sense that the particles are small relative to the separation distances between them. To capture this, we introduce the parameter $\rho$ to capture the radii of the reference particles $D_1,\dots,D_N$. We define $\rho=\tfrac{1}{2}\max_{i} (\text{diam}(D_i))$ where $\text{diam}(D_i)$ is defined as
\begin{equation}
    \text{diam}(D_i)=\sup\{  |x-y|:x,y\in D_i  \}.
\end{equation}
Then, in the case that $\rho$ is small, we have the following lemma, which will be used later.

\begin{lemma}\label{approx}
For all $i=1,\dots,N$ and for characteristic size $\d$ of the same order as $\rho$, we can write that
\begin{align}\label{approx formula}
    u_i = \langle u,\phi_i^{(\d)} \rangle \phi_i^{(\d)} + O(\r^2),
\end{align}
as $\r\to0$, where $\phi_i^{(\d)}$ denotes the eigenvector associated to the particle $D_i$ of the potential $M_{D_i}^{\d k_0}$ and $\r>0$ denotes the particle size parameter of $D_1,\dots,D_N$. Here, $\d$ and $\r$ are of the same order in the sense that $\d=O(\r)$ and $\r=O(\d)$. In this case, the error term holds uniformly for any small $\d$ and $\r$ in a neighbourhood of 0. 
\end{lemma}
\begin{proof}
We refer to Appendix~\ref{app:lem2.5}.
\end{proof}

We can now state the main result in the two-dimensional case.

\begin{theorem} \label{thm:approximation}
The scattering resonance problem in two dimensions becomes, at leading order as $\d\to0$ and $\r\to0$, with $\d=O(\r)$ and $\r=O(\d)$, finding $\w\in\mathbb{C}$ such that
\begin{align*}
    \det(\mathcal{L}) = 0,
\end{align*}
where the matrix $\mathcal{L}$ is given by
\begin{align}\label{L}
    \mathcal{L}_{ij} = 
    \begin{cases}
    \langle N^{\d k_0}_{D_i D_{i+1 \lfloor N \rfloor }} \phi_i^{(\d)}, \phi^{(\d)}_{i+1\lfloor N \rfloor} \rangle , &\text{ if $i=j$,} \\
    - \mathcal{B}_i(\w,\d) \langle N^{\d k_0}_{D_j D_i} \phi^{(\d)}_j, \phi^{(\d)}_i \rangle \langle N^{\d k_0}_{D_i D_{i+1 \lfloor N \rfloor }} \phi_i^{(\d)}, \phi^{(\d)}_{i+1\lfloor N \rfloor} \rangle, &\text{ if $i \ne j$.}
    \end{cases}
\end{align}
Here, $k_0=\mu_0\varepsilon_0\omega$ and
\begin{align}\label{B_i}
    \mathcal{B}_i(\w,\d) := \frac{\d^2 \w^2 \xi(\w,k)}{1 - \d^2 \w^2 \xi(\w,k) \n^{(i)}_{\d}}, \ \ i=1,2,...,N,
\end{align}
with $\n_{\d}^{(i)}$ and $\phi_i^{(\d)}$ being the eigenvalues and the respective eigenvectors associated to the particle $D_i$ of the potential $M^{\d k_0}_{D_i}$, for $i=1,2,\dots,N$.
\end{theorem}

\begin{proof}
We observe that the integral formulation (\ref{2dscat}) is equivalent to
\begin{align}\label{2dscat2}
    \begin{pmatrix}
    u_1 \\ u_2 \\ \vdots \\ u_N
    \end{pmatrix} 
    - \d^2 \w^2 \xi(\w,k) \mathbb{M}
    \begin{pmatrix}
    \displaystyle\sum_{j=1,j\ne1}^N N^{\d k_0}_{D_j D_1} u_j \\
    \displaystyle\sum_{j=1,j\ne2}^N N^{\d k_0}_{D_j D_2} u_j \\
    \vdots \\
    \displaystyle\sum_{j=1,j\ne N}^N N^{\d k_0}_{D_j D_N} u_j
    \end{pmatrix} 
    =
    \begin{pmatrix}
    0 \\ 0 \\ \vdots \\ 0
    \end{pmatrix},
\end{align}
where $\mathbb{M}$ is the diagonal matrix given by
\begin{align*}
    \mathbb{M}_{ij} = 
    \begin{cases}
    \Big(1 - \d^2 \w^2 \xi(\w,k) M^{\d k_0}_{D_i} \Big)^{-1}, \ \ &\text{ if } i=j, \\
    0, &\text{ if } i\ne j,
    \end{cases}
\end{align*}
for $i,j = 1,\dots,N$. From the pole-pencil decomposition, for $i=1,2,\dots,N$, we have
\begin{align*}
    \Big(1 - \d^2 \w^2 \xi(\w,k) M^{\d k_0}_{D_i}\Big)^{-1}(\cdot) = \frac{ \langle \cdot, \phi_i^{(\d)} \rangle \phi_i^{(\d)} }{1 - \d^2 \w^2 \xi(\w,k) \n^{(i)}_{\d}} + R_i[\w](\cdot).
\end{align*}
We recall that, from \cite{AD}, the remainder term $R_i[\w](\cdot)$ can be neglected. Thus, (\ref{2dscat2}) gives
\begin{align*}
    \begin{pmatrix}
    u_1 \\ u_2 \\ \vdots \\ u_N
    \end{pmatrix} 
    - \d^2 \w^2 \xi(\w,k) \tilde{\mathbb{M}}
    \begin{pmatrix}
    \displaystyle\sum_{j=1,j\ne1}^N N^{\d k_0}_{D_j D_1} u_j \\
    \displaystyle\sum_{j=1,j\ne2}^N N^{\d k_0}_{D_j D_2} u_j \\
    \vdots \\
    \displaystyle\sum_{j=1,j\ne N}^N N^{\d k_0}_{D_j D_N} u_j
    \end{pmatrix} 
    =
    \begin{pmatrix}
    0 \\ 0 \\ \vdots \\ 0
    \end{pmatrix},
\end{align*}
where $\tilde{\mathbb{M}}$ is the diagonal matrix given by
\begin{align*}
    \tilde{\mathbb{M}}_{ij} = 
    \begin{cases}
    \frac{ \langle \cdot, \phi_i^{(\d)} \rangle \phi_i^{(\d)} }{1 - \d^2 \w^2 \xi(\w,k) \n^{(i)}_{\d}} , \ \ &\text{ if } i=j, \\
    0, &\text{ if } i \ne j.
    \end{cases}
\end{align*}
This is equivalent to the following system
\begin{align}\label{2dscat3}
    u_i - \frac{\d^2 \w^2 \xi(\w,k)}{1 - \d^2 \w^2 \xi(\w,k) \n^{(i)}_{\d}} \displaystyle\sum\limits_{j=1,j \ne i}^N \langle N^{\d k_0}_{D_j D_i} u_j, \phi_i^{(\d)}  \rangle \phi_i^{(\d)}  = 0, \ \ \ \text{for each } i=1,\dots,N.
\end{align}
Then, applying the operator $N^{\d k_0}_{D_i D_{i+1 \lfloor N \rfloor}}$ to \eqref{2dscat3} for each $i$ and taking the product with $\phi^{(\d)}_{i+1 \lfloor N \rfloor}$, gives
\begin{align}\label{2dscat4}
    \langle N^{\d k_0}_{D_i D_{i+1\lfloor N \rfloor}} u_i, \phi_{i+1\lfloor N \rfloor}^{(\d)}  \rangle - \mathcal{B}_i(\w,\d) \displaystyle\sum\limits_{j=1,j \ne i}^N \langle N^{\d k_0}_{D_j D_i} u_j, \phi_i^{(\d)}  \rangle \langle N^{\d k_0}_{D_i D_{i+1\lfloor N \rfloor}} \phi_i^{(\d)}, \phi_{i+1\lfloor N \rfloor}^{(\d)}  \rangle  = 0, 
\end{align}
for each $i=1,\dots,N$. We observe that for $j=1,\dots, N$, from Lemma~\ref{approx}, the following approximation formula holds
\begin{align*}
    u_j \simeq \langle u, \phi_j^{(\d)} \rangle \phi_j^{(\d)}.
\end{align*}
Applying this to (\ref{2dscat4}), we get
\begin{align}
    \langle N^{\d k_0}_{D_i D_{i+1\lfloor N \rfloor}} \phi_i^{(\d)}, \phi_{i+1\lfloor N \rfloor}^{(\d)}  \rangle \langle u, \phi_i^{(\d)} \rangle - \mathcal{B}_i(\w,\d) \displaystyle\sum\limits_{j=1,j \ne i}^N \langle N^{\d k_0}_{D_j D_i} \phi_j^{(\d)}, \phi_i^{(\d)}  \rangle \langle N^{\d k_0}_{D_i D_{i+1\lfloor N \rfloor}} \phi_i^{(\d)}, \phi_{i+1\lfloor N \rfloor}^{(\d)}  \rangle \langle u, \phi_j^{(\d)} \rangle  = 0,
\end{align}
for each $i=1,\dots,N$. This system has the matrix representation
\begin{align} 
\mathcal{L}
\begin{pmatrix}
    \langle u, \phi_1^{(\d)} \rangle \\
    \langle u, \phi_2^{(\d)} \rangle \\
    \vdots \\
    \langle u, \phi_N^{(\d)} \rangle
\end{pmatrix}
=
\begin{pmatrix}
    0 \\
    0 \\
    \vdots \\
    0
\end{pmatrix},
\end{align}
where $\mathcal{L}$ is given by (\ref{L}), which is the desired result.
\end{proof}

\begin{corollary}\label{indic}
For $i=1,\dots,N$, we can write
$$
\phi_i^{(\d)} = \hat{1}_{D_i},
$$
where $\hat{1}_{D_i} = \frac{1_{D_i}}{\sqrt{|D_i|}}$ and $|D_i|$ is used to denote the volume of $D_i$. This also implies that for $1\leq i,j \leq N$,
\begin{align} \label{2dscat6}
    \mathcal{L}_{ij} = 
    \begin{cases}
    \langle N^{\d k_0}_{D_i D_{i+1\lfloor N \rfloor}} \hat{1}_{D_i}, \hat{1}_{D_{i+1\lfloor N \rfloor}} \rangle, \ \ \ &\text{ if } j=i, \\
    - \mathcal{B}_i(\w,\d) \langle N^{\d k_0}_{D_i D_{i+1\lfloor N \rfloor}} \hat{1}_{D_i}, \hat{1}_{D_{i+1\lfloor N \rfloor}} \rangle^2, &\text{ if } j=i+1\lfloor N \rfloor, \\
    - \mathcal{B}_i(\w,\d) \langle N^{\d k_0}_{D_i D_j} \hat{1}_{D_i}, \hat{1}_{D_j} \rangle \langle N^{\d k_0}_{D_i D_{i+1 \lfloor N \rfloor }} \hat{1}_{D_i}, \hat{1}_{D_{i+1\lfloor N \rfloor}} \rangle, &\text{ otherwise}.
    \end{cases}
\end{align}
\end{corollary}
\begin{proof}
    We have that the eigenvectors $\phi_i^{(\d)}$, $i=1,2,...,N$, given the asymptotic expansion of the operator $M^{\d k_0}_{D_i}$, can be approximated $\phi_i^{(\d)} = \hat{1}_{D_i} + O(\frac{1}{\log(\d)})$, where $\hat{1}_{D_i} = \frac{1_{D_i}}{\sqrt{|D_i|}}$. Then, we can directly see the symmetry argument
    $$
    \langle N^{\d k_0}_{D_i D_j} \hat{1}_{D_i}, \hat{1}_{D_j}  \rangle = \langle N^{\d k_0}_{D_j D_i} \hat{1}_{D_j}, \hat{1}_{D_i} \rangle.
    $$
    This implies that
    $$
    \mathcal{L}_{i,i+1\lfloor N \rfloor} = - \mathcal{B}_i(\w,\d)  \langle N^{\d k_0}_{D_i D_j} \hat{1}_{D_i}, \hat{1}_{D_{i+1}\lfloor N \rfloor} \rangle^2,
    $$
    which gives the desired result.
\end{proof}

\section{Computation of the coupled resonant frequencies} \label{sec:frequencies}

In Theorem~\ref{thm:approximation}, we have derived an asymptotic formula for the resonant frequencies. This amounts to finding the $\w$ such that $\det(\mathcal{L}(\omega)) = 0$. In this section, we will show how to use this asymptotic formula to calculate the resonant frequencies for physical examples. This calculation is not straightforward, since in the integral operators have highly non-linear dependence on $\omega$. However, an explicit formula can be derived under an additional assumption. Furthermore, Muller's method can be used to find the the frequencies for which the coefficient matrix is singular, given appropriate initial guesses.

\subsection{Example: Three circular resonators}

\begin{figure}
\begin{center}
\includegraphics[scale=0.6]{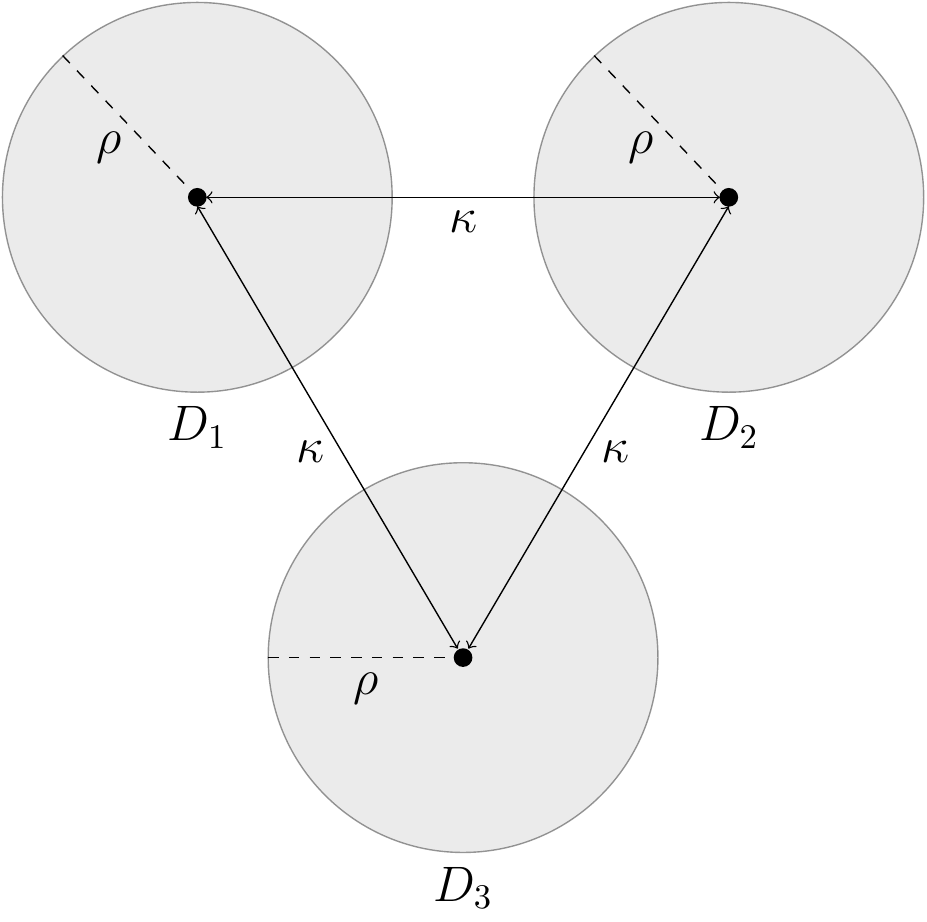}  
\end{center}
\caption{A system of three identical circular resonators can be modelled concisely using our asymptotic method. We study halide perovskite resontators $D_1, D_2$ and $D_3$ of radius $\rho$, made from the same material, with centers placed at a distance $\k$ from each other. } \label{fig:threepart}
\end{figure}

Let us consider the case of having three identical circular halide perovskite resonators $D_1, D_2$ and $D_3$. We will assume that the particles are placed at the same distance $\k$ from each other. This geometry is sketched in Figure~\ref{fig:threepart} and will serve as a suitable example to demonstrate our method. In order to ease the notation, let us write
\begin{align*}
    N_{12}(\w,\d) := \langle N^{\d k_0}_{D_1 D_2} \hat{1}_{D_1},\hat{1}_{D_2} \rangle , \ \ \ \ \
    N_{23}(\w,\d) := \langle N^{\d k_0}_{D_2 D_3} \hat{1}_{D_2},\hat{1}_{D_3} \rangle , \ \ \ \ \
    N_{31}(\w,\d) := \langle N^{\d k_0}_{D_3 D_1} \hat{1}_{D_3},\hat{1}_{D_1} \rangle.
\end{align*}

In order to accelerate the numerical computations and facilitate explicit analytic results, we will make an additional assumption. This assumption is that $N_{ij}(\omega,\delta)$ has no a priori dependence on the frequency $\w$. This is justified in the specific case of halide perovskite nano-particles since $\ve_0$ is of the same magnitude as the characteristic size $\d$. Further, since we are working with the frequencies of the visible light, it holds that $\w$ is of the same magnitude as $\d^{-2}$. Thus, it is reasonable to assume that $\delta k_0$ is constant with respect to $\omega$. Since the dependence of $N_{ij}$ on $\omega$ always takes this form, we can assume it to be approximately independent of $\omega$. We will make this assumption for the results presented in this subsection, and it will be of great importance in studying the inverse design problem in the following section. We write $N_{12}(\w,\d)=N_{12}(\d), N_{23}(\w,\d)=N_{23}(\d)$ and $N_{31}(\w,\d)=N_{31}(\d)$.
Also, since the resonators are identical, it means that they are made from the same material and have the same symmetry. As a result, it holds that $\mathcal{B}_1(\w,\d) = \mathcal{B}_2(\w,\d) = \mathcal{B}_3(\w,\d) =: \mathcal{B}(\w,\d)$. Thus, the matrix $\mathcal{L}$ can be rewritten as
\begin{align}
    \mathcal{L} = 
    \begin{pmatrix}
        N_{12}(\d) & - \mathcal{B}(\w,\d) N_{12}(\d)^2 & - \mathcal{B}(\w,\d)N_{12}(\d)N_{31}(\d)\\
        - \mathcal{B}(\w,\d)N_{12}(\d)N_{23}(\d) & N_{23}(\d) & - \mathcal{B}(\w,\d) N_{23}(\d)^2 \\
        - \mathcal{B}(\w,\d)N_{31}(\d)^2 & - \mathcal{B}(\w,\d)N_{23}(\d)N_{31}(\d) & N_{31}(\d) \\
    \end{pmatrix}.
\end{align}
Then, seeking $\omega$ such that $\det(\mathcal{L}) = 0$, gives that
\begin{align} \label{eq:B}
    2 N_{12}(\d) N_{23}(\d) N_{31}(\d) \mathcal{B}(\w,\d)^3 + \Big(N_{12}(\d)^2+N_{23}(\d)^2+N_{31}(\d)^2\Big) \mathcal{B}(\w,\d)^2 - 1 = 0.
\end{align}
We solve \eqref{eq:B} for $\mathcal{B}(\w,\d)$ and denote the three solutions by $\mathbb{B}_i$, for $i=1,2,3$. Then, solving for $\w\in\mathbb{C}$ in \eqref{B_i}, we have
\begin{align*}
    \Big[ \m_0\a\d^2 + \mathbb{B}_i + \mathbb{B}_i\m_0\a\d^2\n(\d) \Big]\w^2 + i\mathbb{B}_i\g\w - \mathbb{B}_i\b - \mathbb{B}_i \h k^2 = 0,
\end{align*}
from which we obtain
\begin{align}
    \w_i = \frac{-i\mathbb{B}_i\g \pm \sqrt{-\mathbb{B}_i^2\g^2 + 4\Big(\mathbb{B}_i\b+\mathbb{B}_i\h k^2\Big)\Big(\m_0\a\d^2 + \mathbb{B}_i + \mathbb{B}_i\m_0\a\d^2\n(\d)\Big)}}{2\Big(\m_0\a\d^2 + \mathbb{B}_i + \mathbb{B}_i\m_0\a\d^2\n(\d)\Big)}, \ \ \ \ \ \ \ i=1,2,3.
\end{align}

\begin{figure}
\centering
\includegraphics[width=0.8\linewidth]{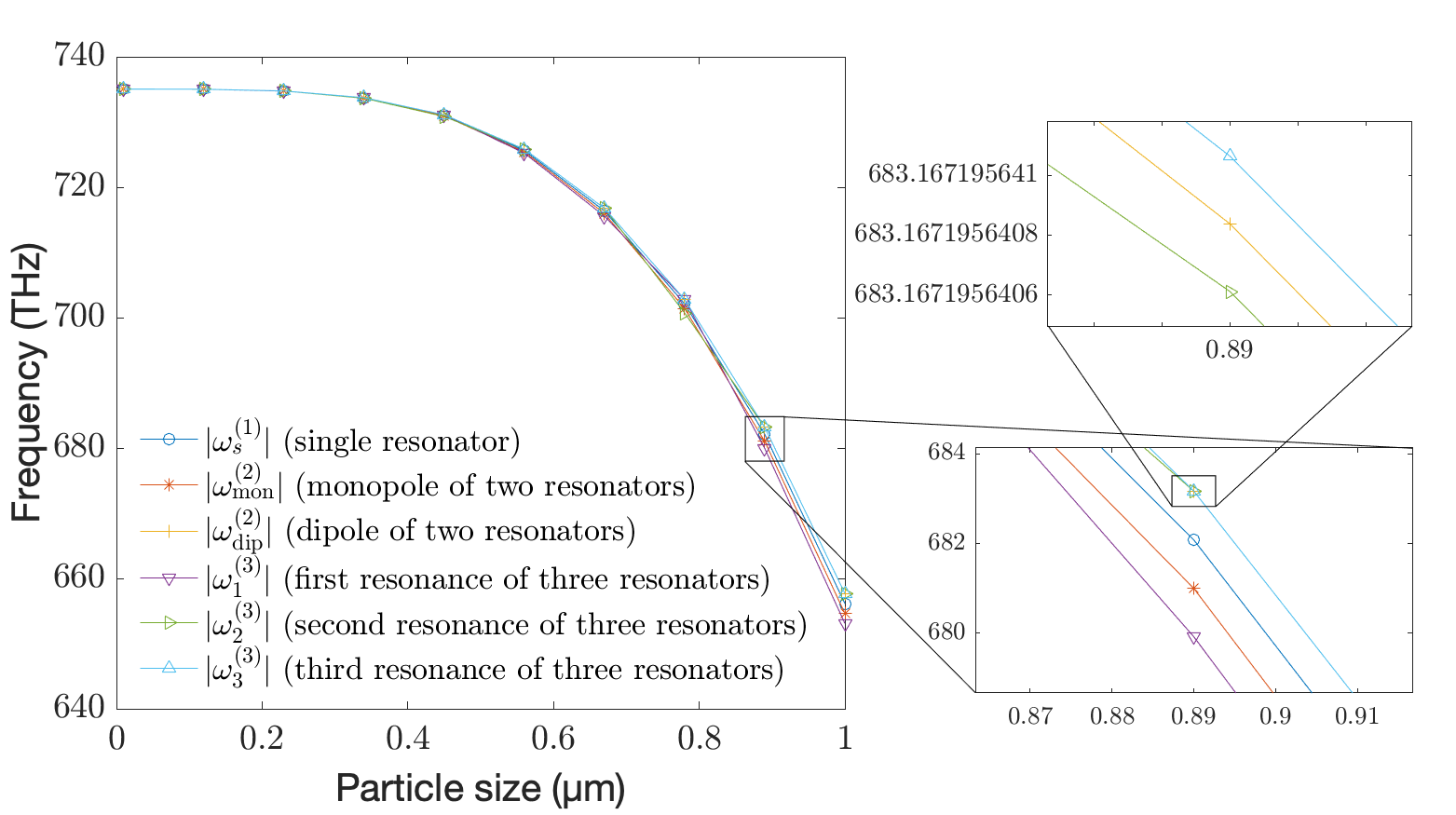}
\caption{Behaviour of the subwavelength resonances for small circular nano-particles of radius $\d$. For three circular methylammonium lead chloride nano-particles, we see how the hybridization causes the frequencies $\w_1^{(3)}, w_2^{(3)}$ and $\w_3^{(3)}$ to shift, relative to the uncoupled resonant frequency of a single particle. We compare them with the hybridized frequencies $\w^{(2)}_\mathrm{mon}, \w^{(2)}_\mathrm{dip}$ of the two circular particle case and the resonant frequency $\w^{(1)}_s$ of the single particle. All the resonators are identical, in the sense that they are the same size and made from the same material (methylammonium lead chloride).} \label{fig:numerics}
\end{figure}

It is helpful to illustrate these results by comparing the case of three resonators to one- and two-particle systems. We plot all these frequencies as a function of the particle size in Figure~\ref{fig:numerics}. The resonant frequency for one particle is denoted by $\w_s^{(1)}$ and the subwavelength frequencies for the case of two particles will be denoted by $\w^{(2)}_{\text{mon}}$ and $\w_{\text{dip}}^{(2)}$. These systems were explored in detail in \cite{AD}, where it was shown that that $\w^{(2)}_{\text{mon}}<\w^{(1)}_s<\w^{(2)}_{\text{dip}}$ as a result of the hybridization. For the case of three particles, we denote the frequencies by $\w_1^{(3)}, \w_2^{(3)}$ and $\w_3^{(3)}$, and we observe that there is also an ordering between them $\w_1^{(3)}<\w_2^{(3)}<\w_3^{(3)}$. Parameter values are chosen to corresponding to methylammonium lead chloride ($\text{MAPbCl}_3$), which is a popular halide perovskite \cite{MFTHBPZK}. We notice that the resonant frequencies for these resonators lies in the range of visible frequencies, when the particles are hundreds of nanometres in size. This puts the system in the appropriate subwavelength regime that was required for our asymptotic method. As $\d\rightarrow0$, the frequencies of the different cases converge to $\w_s^{(1)}$. This is because the nano-particles behave as isolated, identical resonators when $\d$ is very small. Then, as $\d$ increases, we observe that there is a separation between the frequencies of the two particle and three particle case $\w_1^{(3)}<\w^{(2)}_{\text{mon}}<\w_2^{(3)}<\w^{(2)}_{\text{dip}}<\w_3^{(3)}$. In Figures~\ref{fig:numerics}(b) and \ref{fig:numerics}(c), we can see more clearly this separation. This is the effect of the hybridization on the system of resonators.

\section{Inverse design} \label{sec:inverse}

In this section, we will use our asymptotic results to tackle an inverse design problem. Let us assume that we are given three identical two-dimensional circular halide perovskite resonators $D_1,D_2$ and $D_3$ of radius $\r\in\mathbb{R}_{>0}$ and three frequencies $\w_1,\w_2,\w_3\in\mathbb{C}$. Again, we will assume that we are working with nano-particles and that the frequencies given are of the visible light, and so of order $\d^{-2}$. We want to find the appropriate geometry such that the system of three particles resonates at $\w_1$, $\w_2$ and $\w_3$. This toy problem is inspired by human vision, which is sensitive to three different colours, and the desire to design systems capable of giving colour perception to bioinspired artificial eyes made from halide perovskites \cite{GPLLZZSQKJ, lee2018bioinspired}.

\begin{figure}
\begin{center}
\includegraphics[scale=0.6]{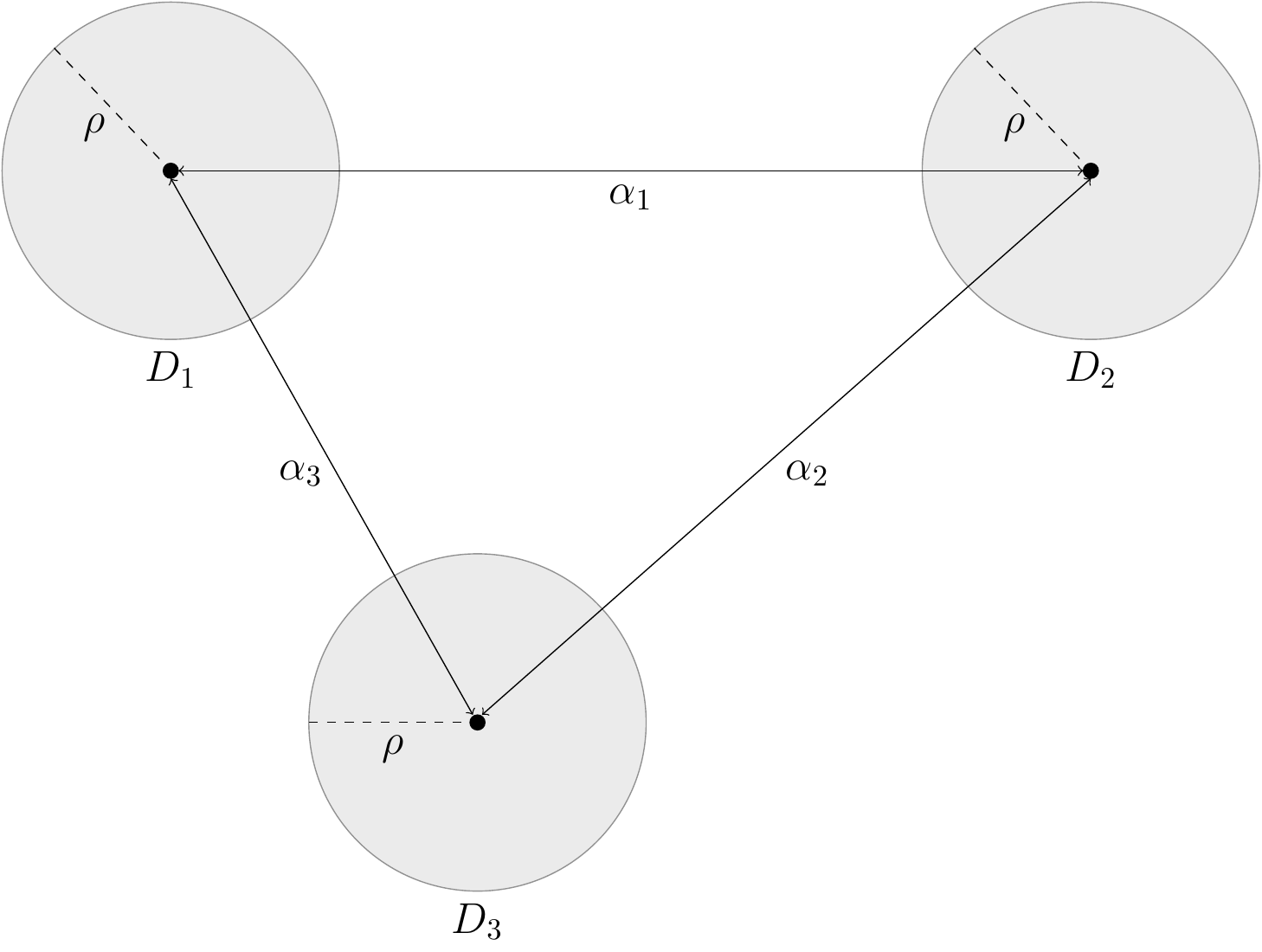}  
\end{center}
\caption{Our asymptotic results can be used to design a system of three identical resonators with specific resonant frequencies. We study a system of three identical circular halide perovskite resontators $D_1, D_2$ and $D_3$, each with radius $\rho$, with centers placed at distances $\mathrm{dist}(D_1,D_2)=\a_1, \mathrm{dist}(D_2,D_3)=\a_2$ and $\mathrm{dist}(D_1,D_3)=\a_3$. } \label{fig:distancesketch}
\end{figure}

Let us denote the separation distances as
\begin{align*}
    \a_1 = \mathrm{dist}(D_1,D_2), \ \ \ 
    \a_2 = \mathrm{dist}(D_2,D_3), \ \ \ 
    \a_3 = \mathrm{dist}(D_1,D_3).
\end{align*}
The configuration is sketched in Figure~\ref{fig:distancesketch}. Then, our problem is finding $\a_1,\a_2,\a_3\in\mathbb{R}$, such that
\begin{align*}
    \det(\mathcal{L})(\w_1,\d) = \det(\mathcal{L})(\w_2,\d) = \det(\mathcal{L})(\w_3,\d) = 0,
\end{align*}
where $\mathcal{L}$ is the coefficient matrix given by (\ref{2dscat6}). This translates into finding $(\a_1,\a_2,\a_3)\in\mathbb{R}^3$, such that
\begin{equation}\label{d-system1}
    \left\{
    \begin{aligned}
    & 2 \mathcal{B}(\w_1,\d)^3 N_{12}(\d) N_{23}(\d) N_{13}(\d) + \mathcal{B}(\w_1,\d)^2 \Big(N_{12}(\d)^2 + N_{23}(\d)^2 + N_{13}(\d)^2\Big) - 1 = 0, \\
    & 2 \mathcal{B}(\w_2,\d)^3 N_{12}(\d) N_{23}(\d) N_{13}(\d) + \mathcal{B}(\w_2,\d)^2 \Big(N_{12}(\d)^2 + N_{23}(\d)^2 + N_{13}(\d)^2\Big) - 1 = 0, \\
    & 2 \mathcal{B}(\w_3,\d)^3 N_{12}(\d) N_{23}(\d) N_{13}(\d) + \mathcal{B}(\w_3,\d)^2 \Big(N_{12}(\d)^2 + N_{23}(\d)^2 + N_{13}(\d)^2\Big) - 1 = 0. 
    \end{aligned}
    \right.
\end{equation}

Our design strategy will have two steps. First, we will find the appropriate characteristic size in order for \eqref{d-system1} to admit a solution. Then, we derive the condition on the separation distances that is required to give the desired resonant frequencies.

\subsection{Linearity of off-diagonal entries}

In order to handle \eqref{d-system1}, it will be helpful to establish how the coefficients $N_{ij}$ depend on the distances between the particles, in the case that $\delta$ is small. Let us first show the following lemma which we will use later and is a consequence of working with small, circular particles.

\begin{lemma}\label{taking distance out}
    Let $d\in\mathbb{R}$ be fixed and $A\in\mathbb{C}$ be given by $A =  R \cos(t) + iR\sin(t)$, where $R,t\in\mathbb{R}$. Then, as $R\rightarrow0$, we have that
    $$
    |A + d| = |A| + d + O(R).
    $$
\end{lemma}
\begin{proof}
    We observe that
    \begin{align*}
        |A+d|^2 &= |R \cos(t) + iR\sin(t) + d|^2 \\ 
        &= R^2 \cos^2(t) + 2 d R \cos(t) + d^2 + R^2 sin^2(t)
    \end{align*}
    and
    \begin{align*}
        \Big( |A|+d \Big)^2 &= R^2 \cos^2(t) + R^2\sin^2(t) + 2 d |R \cos(t) + i R \sin(t)| + d^2. 
    \end{align*}
    Hence, as $R\rightarrow0$,
    \begin{align*}
        |A+d|^2 = \Big( |A|+d \Big)^2 + O(R),
    \end{align*}
    which gives the desired result.
\end{proof}

We will now state a fundamental result which contributes a lot to the analysis of the system.

\begin{theorem}\label{linearity of N}
    There exists $\mathbb{S}=\mathbb{S}(\delta),\mathbb{Q}=\mathbb{Q}(\delta)\in\mathbb{C}$, such that as $\delta\to0$
    \begin{align*}
        N_{ij}(\d) = \mathbb{S} + \mathbb{Q} \mathrm{dist}(D_i,D_j) + O(\d^4), \quad i,j=1,\dots,N,
    \end{align*}
    where $\mathrm{dist}(D_i,D_j)$ denotes the distance between the unscaled particles $D_i$ and $D_j$, which does not depend on $\delta$.
\end{theorem}

\begin{proof}
Let us recall that
\begin{align*}
N_{ij}(\d) &= \langle N^{\d k_0}_{D_i D_j} \hat{1}_{D_i}, \hat{1}_{D_j} \rangle \\
&= \langle \hat{K}^{\d k_0}_{D_i D_j} \hat{1}_{D_i}, \hat{1}_{D_j} \rangle + \langle R^{(0)}_{D_i D_j} \hat{1}_{D_i}, \hat{1}_{D_j} \rangle + (\d k_0)^2 \log(\d k_0 \hat{\g})\langle R^{(1)}_{D_i D_j} \hat{1}_{D_i}, \hat{1}_{D_j} \rangle.
\end{align*}
We will look at this expression term by term. We observe that
\begin{align*}
    \langle \hat{K}^{\d k_0}_{D_i D_j} \hat{1}_{D_i}, \hat{1}_{D_j} \rangle = 
    - \frac{1}{2\pi} \log(\d k_0 \hat{\g}) \int_{D_j} \int_{D_i} \hat{1}_{D_i}(y) \upd y \hat{1}_{D_j}(x) \upd x.
\end{align*}
Since there is no distance element appearing in the integrand, there is not dependence on the distance between the particles $D_i$ and $D_i$. Thus, this is a constant with respect to the resonator distance,
\begin{align}\label{K_ij}
    K_{ij} := \langle \hat{K}^{\d k_0}_{D_i D_j} \hat{1}_{D_i}, \hat{1}_{D_j} \rangle
\end{align}
Next, we have
\begin{align*}
    \langle R^{(0)}_{D_i D_j}& \hat{1}_{D_i}, \hat{1}_{D_j} \rangle = -\frac{1}{2\pi} \int_{D_j} \int_{D_i} \log|x-y| \hat{1}_{D_i}(y) \upd y \hat{1}_{D_j}(x) \upd x \\
    &= - \frac{1}{2\pi\sqrt{|D_i||D_j|}} \int_{0}^{2\pi} \int_{0}^{\r} \int_{0}^{2\pi} \int_{0}^{\r} \log\Big|r_x e^{i t_x} - r_y e^{i t_y} + \mathrm{dist}(D_i,D_j)\Big| r_y r_x \upd r_y \upd t_y \upd r_x \upd t_x
\end{align*}
where we have changed to polar coordinates and used the fact that the particles are circular and identical. From this, we also get $|D_1|=|D_2|=|D_3|=\pi \r^2$. In addition, using the Taylor expansion of the logarithm function and Lemma \ref{taking distance out}, we have
\begin{align*}
    \log\Big|r_x e^{i t_x} - r_y e^{i t_y} + \text{dist}(D_i,D_j)\Big| \simeq \Big|r_x e^{i t_x} - r_y e^{i t_y}\Big| + \text{dist}(D_i,D_j) - 1 + O(\r^2).
\end{align*}
If we define
\begin{equation*}
    R^{(0)}:=- \frac{1}{2\pi^2\r^2} \int_{0}^{2\pi} \int_{0}^{\r} \int_{0}^{2\pi} \int_{0}^{\r} \Big(\Big|r_x e^{i t_x} - r_y e^{i t_y}\Big| - 1 \Big) r_y r_x \upd r_y \upd t_y \upd r_x \upd t_x,
\end{equation*}
then we have that
\begin{align} \label{R0_ij}
    \langle R^{(0)}_{D_i D_j} \hat{1}_{D_i}, \hat{1}_{D_j} \rangle = 
    R^{(0)} - \frac{\r^2}{2} \text{dist}(D_i,D_j) + O(\r^4).
\end{align}
The last term can be rewritten as
\begin{align*}
    \langle R^{(1)}_{D_i D_j}& \hat{1}_{D_i}, \hat{1}_{D_j} \rangle = - \frac{i}{4\pi} \int_{D_j} \int_{D_i} \frac{1}{|x-y|} \hat{1}_{D_i}(y) \upd y \hat{1}_{D_j}(x) \upd x \\
    &= \frac{-i}{4\pi^2\r^2} \int_{0}^{2\pi} \int_{0}^{\r} \int_{0}^{2\pi} \int_{0}^{\r} \frac{r_y r_x}{\Big|r_x e^{i t_x} - r_y e^{i t_y} + \text{dist}(D_i,D_j)\Big|} \upd r_y \upd t_y \upd r_x \upd t_x
\end{align*}
Again, using the Taylor expansion and Lemma \ref{taking distance out}, we have 
\begin{align*}
    \frac{1}{\Big|r_x e^{i t_x} - r_y e^{i t_y} + \text{dist}(D_i,D_j)\Big|} \simeq 2 + \text{dist}(D_i,D_j) - \Big|r_x e^{i t_x} - r_y e^{i t_y}\Big| + O(\r^2)
\end{align*}
Hence, defining
\begin{equation*}
    R^{(1)} := \frac{-i}{4\pi^2\r^2} \int_{0}^{2\pi} \int_{0}^{\r} \int_{0}^{2\pi} \int_{0}^{\r} \Big( 2 - \Big|r_x e^{i t_x} - r_y e^{i t_y}\Big| \Big) r_y r_x \upd r_y \upd t_y \upd r_x \upd t_x,
\end{equation*}
gives
\begin{align}\label{R1_ij}
        \langle R^{(1)}_{D_i D_j} \hat{1}_{D_i}, \hat{1}_{D_j} \rangle 
        = R^{(1)} - \frac{i \r^2}{4} \text{dist}(D_i, D_j) + O(\r^4).
\end{align}
Gathering the results (\ref{K_ij}), (\ref{R0_ij}) and (\ref{R1_ij}), we obtain
\begin{align}
    N_{ij}(\d) = K_{ij} + R^{(0)} + (\d k_0)^2 \log(\d k_0 \hat{\g}) R^{(1)} + \left[ -\frac{\r^2}{2} - \frac{i \r^2}{4}(\d k_0)^2 \log(\d k_0 \hat{\g}) \right] \text{dist}(D_i,D_j) + O(\d^4),
\end{align}
and thus, by defining
\begin{align*}
    \mathbb{S}_{ij} := K_{ij} + R^{(0)} + (\d k_0)^2 \log(\d k_0 \hat{\g}) R^{(1)}
\end{align*}
and
\begin{align*}
    \mathbb{Q} := -\frac{\r^2}{2} - \frac{i \r^2}{4}(\d k_0)^2 \log(\d k_0 \hat{\g}),
\end{align*}
we get
\begin{align*}
    N_{ij}(\d) = \mathbb{S}_{ij} + \mathbb{Q}\text{dist}(D_i,D_j) + O(\d^4).
\end{align*}
Since the particles are identical, we have directly that $\mathbb{S}_{12} = \mathbb{S}_{23} = \mathbb{S}_{13} =: \mathbb{S}$, from which the result follows.
\end{proof}

\begin{remark}
We note that this theorem can also be generalized to the cases where the resonators are not circular. The adaptation required would be a change in the definitions of $\mathbb{S}$ and $\mathbb{Q}$.
\end{remark}

The above theorem allows us to write
\begin{align}\label{N(a)}
    N_{12} = \mathbb{S} + \mathbb{Q} \a_1 +  O(\d^4) , \ \ \ \
    N_{23} = \mathbb{S} + \mathbb{Q} \a_2 +  O(\d^4) \ \ \text{ and } \ \  
    N_{13} = \mathbb{S} + \mathbb{Q} \a_3 +  O(\d^4).
\end{align}

\subsection{Condition on characteristic size}

The first thing that we wish to understand is when the system (\ref{d-system1}) has a solution. Let us write
$$
X = N_{12} N_{23} N_{13} \ \ \ \ \ \ \ \ \text{ and } \ \ \ \ \ \ \ \ Y = N_{12}^2 + N_{23}^2 + N_{13}^2.
$$
Then, (\ref{d-system1}) becomes
\begin{equation}\label{d-system2}
    \left\{
    \begin{aligned}
    & 2 \mathcal{B}(\w_1,\d)^3 X + \mathcal{B}(\w_1,\d)^2 Y - 1 = 0, \\
    & 2 \mathcal{B}(\w_2,\d)^3 X + \mathcal{B}(\w_2,\d)^2 Y - 1 = 0, \\
    & 2 \mathcal{B}(\w_3,\d)^3 X + \mathcal{B}(\w_3,\d)^2 Y - 1 = 0. 
    \end{aligned}
    \right.
\end{equation}
Using the Gauss elimination process, we get that the following equation needs to be satisfied
\begin{equation*}
    \mathcal{B}(\w_3,\d)^2 \Big[\mathcal{B}(\w_1,\d)^3-\mathcal{B}(\w_2,\d)^3\Big]\Big[\mathcal{B}(\w_3,\d) - \mathcal{B}(\w_1,\d)\Big] = 
    \mathcal{B}(\w_2,\d)^2 \Big[\mathcal{B}(\w_1,\d)^3-\mathcal{B}(\w_3,\d)^3\Big]\Big[\mathcal{B}(\w_2,\d) - \mathcal{B}(\w_1,\d)\Big].
\end{equation*}
Expanding this, we obtain that the characteristic size $\d$ needs to satisfy 
\begin{align}\label{d-condition}
\begin{aligned}
    \d^2 \n(\d) = &\frac{-1}{3\w_1^2\xi(\w_1,k)\w_2^2\xi(\w_2,k)\w_3^2\xi(\w_3,k)} \cdot \Big[ \w_1^4\xi(\w_1,k)^2\w_3^6\xi(\w_3,k)^3 - \w_1^6\xi(\w_1,k)^3\w_3^4\xi(\w_3,k)^2 + \\
    &+ \w_2^6\xi(\w_2,k)^3\w_3^4\xi(\w_3,k)^2 - \w_1^4\xi(\w_1,k)^2\w_2^6\xi(\w_2,k)^3 + \w_1^6\xi(\w_1,k)^3\w_2^4\xi(\w_2,k)^2 - \\
    &-\w_2^4\xi(\w_2,k)^2\w_3^6\xi(\w_3,k)^3 \Big] \cdot \Big[ \w_1^2\xi(\w_1,k)\w_2^4\xi(\w_2,k)^2 + \w_2^2\xi(\w_2,k)\w_3^4\xi(\w_3,k)^2 +\\
    &+ \w_1^4\xi(\w_1,k)^2\w_3^2\xi(\w_3,k) - \w_1^2\xi(\w_1,k)\w_3^4\xi(\w_3,k)^2 - \w_2^4\xi(\w_2,k)^2\w_3^2\xi(\w_3,k) \\
    &- \w_1^4\xi(\w_1,k)^2\w_2^2\xi(\w_2,k) \Big]^{-1},
\end{aligned}
\end{align}
for (\ref{d-system2}) to have a solution.

\subsection{Condition on separation distances}

We assume that the condition (\ref{d-condition}) is satisfied. Then, we can reduce our study of the system (\ref{d-system2}) to finding a solution to
\begin{equation}\label{d-system3}
    \left\{
    \begin{aligned}
    & 2 \mathcal{B}(\w_1,\d)^3 X + \mathcal{B}(\w_1,\d)^2 Y - 1 = 0, \\
    & 2 \mathcal{B}(\w_2,\d)^3 X + \mathcal{B}(\w_2,\d)^2 Y - 1 = 0. 
    \end{aligned}
    \right.
\end{equation}
This gives
\begin{align}\label{X}
    X = \frac{ \mathcal{B}(\w_2,\d)^2\Big[ \mathcal{B}(\w_2,\d) - \mathcal{B}(\w_1,\d) \Big] - \Big[ \mathcal{B}(\w_2,\d)^3 - \mathcal{B}(\w_1,\d)^3 \Big] }{ 2 \mathcal{B}(\w_1,\d) \mathcal{B}(\w_2,\d)^2 \Big[ \mathcal{B}(\w_2,\d)\mathcal{B}(\w_1,\d)^2 - \mathcal{B}(\w_1,\d)^3 \Big] }
\end{align}
and
\begin{align}\label{Y}
    Y = \frac{\mathcal{B}(\w_2,\d)^3 - \mathcal{B}(\w_1,\d)^3}{\mathcal{B}(\w_1,\d)^2 \mathcal{B}(\w_2,\d)^2 \Big[ \mathcal{B}(\w_2,\d) - \mathcal{B}(\w_1,\d) \Big]}.
\end{align}
Fixing these values for $X$ and $Y$ and varying $\a_3\in\mathbb{R}$, we get from (\ref{d-system1}),
\begin{align}\label{a_2}
    \a_2(\a_3) = \frac{1}{\mathbb{Q}} \left(-\mathbb{S} \pm \sqrt{ \frac{ - C \pm \sqrt{C^2 - 4 X^2 (\mathbb{S}+\mathbb{Q}\a_3)^2} }{2(\mathbb{S}+\mathbb{Q}\a_3)^2} } \right),
\end{align}
where $C = (\mathbb{S}+\mathbb{Q}\a_3)^2 [(\mathbb{S}+\mathbb{Q}\a_3)^2 - Y]$ and
\begin{align}\label{a_1}
    \a_1(\a_3) = \frac{1}{\mathbb{Q}} \left( \frac{X}{\Big(\mathbb{S}+\mathbb{Q}\a_2(\a_3)\Big)\Big(\mathbb{S}+\mathbb{Q}\a_3\Big)} - \mathbb{S} \right).
\end{align}
Let us also note here, that in order for the distances found to make geometric sense, we require
\begin{align}\label{triangle condition}
    \Big|\a_3 - \a_2(\a_3)\Big| \leq \a_1(\a_3) \leq \Big| \a_3 + \a_2(\a_3) \Big|,
\end{align}
which gives an additional condition on $\a_3\in\mathbb{R}$. Therefore, we conclude that the distances $\a_1$, $\a_2$ and $\a_3$ must lie in the one-dimensional space given by
\begin{align}\label{1d-space}
    \left\{
    \begin{pmatrix}
    \a_1(\a_3) \\ \a_2(\a_3) \\ \a_3
    \end{pmatrix}
    : \a_3\in\mathbb{R} \text{ such that } (\ref{triangle condition}) \text{ holds and } \d\in\mathbb{R} \text{ is given by } (\ref{d-condition})
    \right\}.
\end{align}

\section{Conclusion}

We have developed an approach for modelling a coupled system of many subwavelength halide perovskite resonators. Their highly dispersive material parameters makes this a challenging problem, but, given their rapidly growing usage in electromagnetic devices, efficient mathematical methods like ours are becoming increasingly valuable. Our method is sufficiently concise that we have been able to use it for an inverse design problem, which would have required significant computational effort to solve using numerical simulation methods. These results can accelerate the design of advanced photonic devices \cite{JKM,KMBKL}, including those with complicated structures and geometries, such as the biomimetic eye developed by \cite{GPLLZZSQKJ}.

\section*{Conflicts of Interest}

The authors declare no conflicts of interest.

\section*{Acknowledgements}

The work of KA was supported by ETH Z\"urich under the project ETH-34 20-2. The work of BD was supported by the European Research Council H2020 FETOpen project BOHEME under grant agreement No.~863179.

\appendix

\section{Appendix}

\subsection{Three dimensions} \label{app:3d}

Here, we present the fundamentals of the analysis of the problem in the three-dimensional setting. We consider $N\in\mathbb{N}$ halide perovskite resontators $D_1,D_2, \dots, D_N$, made from the same material. We consider the integral operators $K^{\d k_0}_{D_i}$ and $R^{\d k_0}_{D_i D_j}$, for $i,j=1,2,...,N$ defined as in Definition~\ref{def:KR}. Then, the following lemma is a direct consequence of these definitions.

\begin{lemma}
The scattering problem (\ref{repping}) can be restated, using the Definition~\ref{def:KR}, as
\begin{align}\label{system 1}
    \begin{pmatrix}
    1-\d^2\w^2\xi(\w,k)K^{\d k_0}_{D_1} & -\d^2\w^2\xi(\w,k) R^{\d k_0}_{D_2 D_1} & \dots & -\d^2\w^2\xi(\w,k) R^{\d k_0}_{D_N D_1} \\
    -\d^2\w^2\xi(\w,k) R^{\d k_0}_{D_1 D_2} & 1-\d^2\w^2\xi(\w,k)K^{\d k_0}_{D_2} & \dots & -\d^2\w^2\xi(\w,k) R^{\d k_0}_{D_N D_2}\\
    \vdots & \vdots & \ddots & \vdots \\
    -\d^2\w^2\xi(\w,k) R^{\d k_0}_{D_1 D_N} & -\d^2\w^2\xi(\w,k) R^{\d k_0}_{D_2 D_N} & \dots & 1-\d^2\w^2\xi(\w,k)K^{\d k_0}_{D_N}
    \end{pmatrix}
    \begin{pmatrix}
    u|_{D_1} \\ u|_{D_2} \\ \vdots \\ u|_{D_N}
    \end{pmatrix}
    =
    \begin{pmatrix}
    u_{in}|_{D_1} \\ u_{in}|_{D_2} \\ \vdots \\ u_{in}|_{D_N}
    \end{pmatrix}.
\end{align}
\end{lemma}
Thus, the scattering resonance problem is to find $\w$ such that the operator in (\ref{system 1}) is singular, or equivalently, such that there exists $(u_1,u_2,...,u_N) \in L^2(D_1) \times L^2(D_2) \times ... \times L^2(D_N)$, $(u_1,u_2,...,u_N) \ne \mathbf{0}$, such that
\begin{align}\label{system 2}
    \begin{pmatrix}
    1-\d^2\w^2\xi(\w,k)K^{\d k_0}_{D_1} & -\d^2\w^2\xi(\w,k) R^{\d k_0}_{D_2 D_1} & \dots & -\d^2\w^2\xi(\w,k) R^{\d k_0}_{D_N D_1} \\
    -\d^2\w^2\xi(\w,k) R^{\d k_0}_{D_1 D_2} & 1-\d^2\w^2\xi(\w,k)K^{\d k_0}_{D_2} & \dots & -\d^2\w^2\xi(\w,k) R^{\d k_0}_{D_N D_2}\\
    \vdots & \vdots & \ddots & \vdots \\
    -\d^2\w^2\xi(\w,k) R^{\d k_0}_{D_1 D_N} & -\d^2\w^2\xi(\w,k) R^{\d k_0}_{D_2 D_N} & \dots & 1-\d^2\w^2\xi(\w,k)K^{\d k_0}_{D_N}
    \end{pmatrix}
    \begin{pmatrix}
    u_1 \\ u_2 \\ \vdots \\ u_N
    \end{pmatrix}
    =
    \begin{pmatrix}
    0 \\ 0 \\ \vdots \\ 0
    \end{pmatrix}.
\end{align}
This gives the main result of the three-dimensional case.
\begin{theorem}
The scattering resonance problem in three dimensions becomes finding $\w\in\mathbb{C}$, such that
\begin{align*}
    \det(\mathcal{K}) = 0
\end{align*}
where the matrix $\mathcal{K}$ is given by
\begin{align}\label{K}
    \mathcal{K}_{ij} = 
    \begin{cases}
    \langle R^{\d k_0}_{D_i D_{i+1 \lfloor N \rfloor }} \phi_i^{(\d)}, \phi^{(\d)}_{i+1\lfloor N \rfloor} \rangle , &\text{ if $i=j$,} \\
    - \mathcal{A}_i(\w,\d) \langle R^{\d k_0}_{D_j D_i} \phi^{(\d)}_j, \phi^{(\d)}_i \rangle \langle R^{\d k_0}_{D_i D_{i+1 \lfloor N \rfloor }} \phi_i^{(\d)}, \phi^{(\d)}_{i+1\lfloor N \rfloor} \rangle, &\text{ if $i \ne j$.}
    \end{cases}
\end{align}
where $k_0=\mu_0\varepsilon_0\omega$ and
\begin{align}\label{A_i}
    \mathcal{A}_i(\w,\d) := \frac{\d^2 \w^2 \xi(\w,k)}{1 - \d^2 \w^2 \xi(\w,k) \l^{(i)}_{\d}}, \ \ i=1,...,N,
\end{align}
where $\l^{(i)}_{\d}$ and $\phi_i^{(\d)}$ are the eigenvalues and the respective eigenvectors associated to the particle $D_i$ of the potential $K^{\d k_0}_{D_i}$, for $i=1,2,\dots,N$.
\end{theorem}
\begin{proof}
We observe that (\ref{system 2}) is equivalent to 
\begin{align*}
    &\begin{pmatrix}
    1-\d^2\w^2\xi(\w,k)K^{\d k_0}_{D_1} & 0 & \dots & 0 \\
    0 & 1-\d^2\w^2\xi(\w,k)K^{\d k_0}_{D_2} & \dots & 0 \\
    \vdots & \vdots & \ddots & \vdots \\
   0 & 0 & \dots & 1-\d^2\w^2\xi(\w,k)K^{\d k_0}_{D_N}
    \end{pmatrix}
    \begin{pmatrix}
    u_1 \\ u_2 \\ \vdots \\ u_N
    \end{pmatrix} \\
    &\hspace{+3cm}
     - \d^2 \w^2 \xi(\w,k)
    \begin{pmatrix}
    0 &  R^{\d k_0}_{D_2 D_1} & \dots &  R^{\d k_0}_{D_N D_1} \\
     R^{\d k_0}_{D_1 D_2} & 0 & \dots &  R^{\d k_0}_{D_N D_2}\\
    \vdots & \vdots & \ddots & \vdots \\
     R^{\d k_0}_{D_1 D_N} &  R^{\d k_0}_{D_2 D_N} & \dots & 0
    \end{pmatrix}
    \begin{pmatrix}
    u_1 \\ u_2 \\ \vdots \\ u_N
    \end{pmatrix}
    =
    \begin{pmatrix}
    0 \\ 0 \\ \vdots \\ 0
    \end{pmatrix},
\end{align*}
which gives
\begin{align}\label{scat4}
    \begin{pmatrix}
    u_1 \\ u_2 \\ \vdots \\ u_N
    \end{pmatrix} 
    - \d^2 \w^2 \xi(\w,k)
    \mathbb{N}
    \begin{pmatrix}
    \displaystyle\sum_{j=1,j\ne1}^N R^{\d k_0}_{D_j D_1} u_j \\
    \displaystyle\sum_{j=1,j\ne2}^N R^{\d k_0}_{D_j D_2} u_j \\
    \vdots \\
    \displaystyle\sum_{j=1,j\ne N}^N R^{\d k_0}_{D_j D_N} u_j
    \end{pmatrix} 
    =
    \begin{pmatrix}
    0 \\ 0 \\ \vdots \\ 0
    \end{pmatrix},
\end{align}
where $\mathbb{N}$ is the diagonal matrix given by 
\begin{align*}
    \mathbb{N}_{ij} = 
    \begin{cases}
    \Big( 1 - \d^2 \w^2\xi(\w,k) K^{\d k_0}_{D_i} \Big)^{-1} , \ \ &\text{ if } i=j, \\
    0, &\text{ if } i \ne j.
    \end{cases}
\end{align*}
Let us now apply a pole-pencil decomposition on the operators $\Big(1 - \d^2 \w^2 \xi(\w,k) K^{\d k_0}_{D_i}\Big)^{-1}$, for $i=1,2,\dots,N$. We see that
\begin{align*}
    \Big(1 - \d^2 \w^2 \xi(\w,k) K^{\d k_0}_{D_i}\Big)^{-1}(\cdot) = \frac{ \langle \cdot, \phi_i^{(\d)} \rangle \phi_i^{(\d)} }{1 - \d^2 \w^2 \xi(\w,k) \l^{(i)}_{\d}} + R_i[\w](\cdot),
\end{align*}
where $\l^{(i)}_{\d}$ and $\phi_i^{(\d)}$ are the eigenvalues and the respective eigenvectors of the potential $K^{\d k_0}_{D_i}$ associated to the particle $D_i$, for $i=1,2,\dots,N$. We also recall that the remainder terms $R_i[\w](\cdot)$ can be neglected (\cite{AD}). Then, (\ref{scat4}) becomes
\begin{align*}
    \begin{pmatrix}
    u_1 \\ u_2 \\ \vdots \\ u_N
    \end{pmatrix} 
    - \d^2 \w^2 \xi(\w,k)
    \tilde{\mathbb{N}}
    \begin{pmatrix}
    \displaystyle\sum_{j=1,j\ne1}^N R^{\d k_0}_{D_j D_1} u_j \\
    \displaystyle\sum_{j=1,j\ne2}^N R^{\d k_0}_{D_j D_2} u_j \\
    \vdots \\
    \displaystyle\sum_{j=1,j\ne N}^N R^{\d k_0}_{D_j D_N} u_j
    \end{pmatrix} 
    =
    \begin{pmatrix}
    0 \\ 0 \\ \vdots \\ 0
    \end{pmatrix},
\end{align*}
where the matrix $\tilde{\mathbb{N}}$ is given by
\begin{align*}
    \tilde{\mathbb{N}}_{ij} = 
    \begin{cases}
    \frac{ \langle \cdot, \phi_i^{(\d)} \rangle \phi_i^{(\d)} }{1 - \d^2 \w^2 \xi(\w,k) \l^{(i)}_{\d}} , \ \ &\text{ if } i=j, \\
    0, &\text{ if } i \ne j.
    \end{cases}
\end{align*}
This is equivalent to the system of equations
\begin{align*}
    u_i - \frac{\d^2 \w^2 \xi(\w,k)}{1 - \d^2 \w^2 \xi(\w,k) \l^{(i)}_{\d}} \displaystyle\sum\limits_{j=1,j \ne 1}^N \langle R^{\d k_0}_{D_j D_i} u_j, \phi_i^{(\d)}  \rangle \phi_i^{(\d)}  &= 0, \quad \text{ for each } i=1,\dots,N.
\end{align*}
We apply on the $i$-th line the operator $R^{\d k_0}_{D_i D_{i+1 \lfloor N \rfloor}}$ and then take the product with $\phi^{(\d)}_{i+1 \lfloor N \rfloor}$. Then, we find that
\begin{align}\label{slang}
    \langle R^{\d k_0}_{D_i D_{i+1 \lfloor N \rfloor}} u_i, \phi_{i+1 \lfloor N \rfloor}^{(\d)}  \rangle - \frac{\d^2 \w^2 \xi(\w,k)}{1 - \d^2 \w^2 \xi(\w,k) \l^{(i)}_{\d}} \sum\limits_{j=1,j \ne i}^N \langle R^{\d k_0}_{D_j D_i} u_j, \phi_i^{(\d)}  \rangle \langle R^{\d k_0}_{D_i D_{i+1 \lfloor N \rfloor}} \phi_i^{(\d)}, \phi_{i+1 \lfloor N \rfloor}^{(\d)}  \rangle  = 0, 
\end{align}
for each $i=1,\dots,N$. Then, using the definition (\ref{A_i}), the system (\ref{slang}) becomes
\begin{align}\label{profora}
    \langle R^{\d k_0}_{D_i D_{i+1 \lfloor N \rfloor}} u_i, \phi_{i+1 \lfloor N \rfloor}^{(\d)}  \rangle - \mathcal{A}_i(\w,\d) \displaystyle\sum\limits_{j=1,j \ne i}^N \langle R^{\d k_0}_{D_j D_i} u_j, \phi_i^{(\d)}  \rangle \langle R^{\d k_0}_{D_i D_{i+1 \lfloor N \rfloor}} \phi_i^{(\d)}, \phi_{i+1 \lfloor N \rfloor}^{(\d)}  \rangle  = 0,
\end{align}
for each $i=1,\dots,N$. Applying (\ref{approx formula}) to (\ref{profora}), we reach the linear system of equations
\begin{align} 
\mathcal{K}
\begin{pmatrix}
    \langle u, \phi_1^{(\d)} \rangle \\
    \langle u, \phi_2^{(\d)} \rangle \\
    \vdots \\
    \langle u, \phi_N^{(\d)} \rangle
\end{pmatrix}
=
\begin{pmatrix}
    0 \\
    0 \\
    \vdots \\
    0
\end{pmatrix},
\end{align}
where $\mathcal{K}$ is the matrix given by (\ref{K}).
\end{proof}

\subsection{Proof of Lemma~\ref{approx}} \label{app:lem2.5}

\begin{proof}
We will show that the approximation formula (\ref{approx formula}) holds for sufficiently small $\rho\to0$, when $\delta$ is also small. It is important to check the uniformity of these results with respect to $\delta$. In particular, we will take $\r>0$ such that $\r\to0$ at the same rate as $\d\to0$. That is, $\r=O(\d)$ and $\d=O(\r)$. This gives the uniformity of the error term with respect to small characteristic size $\d$.

Our argument is based on Theorem~2.10 of \cite{BLRM}. In particular, once we have shown that the assumptions of this Theorem hold, Lemma~\ref{approx} will follow directly. We will present this proof in the two-dimensional setting, but it could easily be modified to three dimensions. Also, for simplicity, we will consider identical resonators, but the proof will be the same for particles of different sizes.

Recall that in Corollary~\ref{indic} we showed that $\phi_{i}^{(\d)} = \hat{1}_{D_i}$, for $i=1,2,..,N$. As a result, the desired approximation $u_i \simeq \langle u,\phi_i^{(\d)} \rangle \phi_i^{(\d)} + O(\r^2)$ from (\ref{approx formula}) is equivalent to $u_i \simeq \langle u, \hat{1}_{D_i} \rangle \hat{1}_{D_i} + O(\r^2)$. Then, we define the operator $p_{\d}$ as follows
\begin{align}
    p_{\r}: u_i \in L^2(D_i) &\longmapsto \langle u, \hat{1}_{D_i} \rangle \hat{1}_{D_i} \in L^{2}(D),
\end{align}

To be able to use Theorem~2.10 of \cite{BLRM}, the conditions that need to be satisfied are the following:
\begin{enumerate}
    \item It holds that
    \begin{align*}
        \lim_{\r\to0} \| p_{\r}u_i \|_{L^2(D)} = \| u_i \|_{L^2(D)}, \quad \forall i=1,\dots,N.
    \end{align*}
    \item For every compact set $\mathcal{C}\subset\mathbb{C}\setminus\{0\}$, it holds
    \begin{align*}
    \sup_{\w\in\mathcal{C}} \| \mathcal{L} \|_{\sup} < \infty,
    \end{align*}
    uniformly for all $\d>0$ and all $\r>0$, where the norm $\|\cdot\|_{\sup}$ is defined for a square matrix $P \in \mathbb{C}^{N\times N}$ as $\| P \|_{\sup} = \sup_{1\leq i,j \leq N} |P_{ij}|$.
    \item $\langle u, \hat{1}_{D_i} \rangle \hat{1}_{D_i}$ converges regularly to $u_i$, \emph{i.e.}
    \begin{itemize}
        \item $\lim_{\r\to0} \| \mathcal{L}p_{\r}u - \mathcal{F}u \|_{\sup} = 0,$ where $\mathcal{F}u$ denotes our system without the use of the approximation formula (\ref{approx formula}).
        \item For every subsequence $\r'$ of $\r$, it holds that
        $\lim_{\r'\to0} \| u_i - p_{\r'}u_i \|_{L^2(D)} = 0, \quad \forall i=1,\dots,N. $
    \end{itemize}
\end{enumerate}
Let us proceed to their proof.

\subsubsection{First condition: Convergence in norm}

We have that
\begin{align*}
\begin{aligned}
    \|p_{\r}u_i \|_{L^2(D)} &=\left( \int_{D} \left| \int_{D} u(y) \hat{1}_{D_i}(y) \upd y \hat{1}_{D_i}(x) \right|^2 \upd x \right)^{\frac{1}{2}} =\left( \frac{1}{|D_i|^2} \int_{D_i} \left| \int_{D_i} u(y) \upd y \right| \upd x \right)^{\frac{1}{2}}\\
    &= \frac{1}{\sqrt{|D_i|}} \left| \int_{D_i} u(y) \upd y 
    \right|.
\end{aligned}
\end{align*}
Then, from the Cauchy-Schwartz inequality,
\begin{align*}
    \|p_{\r}u_i \|_{L^2(D)} \leq \frac{1}{\sqrt{|D_i|}} \left( \int_{D_i} |u(y)|^2 \upd y \int_{D_i} 1 \upd y \right)^{\frac{1}{2}} = \| u_i \|_{L^2(D)}.
\end{align*}
We can also see that, as $\r\to0$,
\begin{align*}
    \| u_i \|_{L^2(D)} \to 0.
\end{align*}
Hence, we have that
\begin{align}\label{(A1)}
    \lim_{\r\to0} \|p_{\r}u_i \|_{L^2(D)} =  \| u_i \|_{L^2(D)}.
\end{align}

\subsubsection{Second condition: Matrix norm boundedness}

We need to show that, for every compact $\mathcal{C}\subset\mathbb{C}\setminus\{0\}$,
\begin{align}\label{A2}
    \sup_{\w\in\mathcal{C}} \| \mathcal{L} \|_{\sup} = \sup_{\w\in\mathcal{C}} \left( \sup_{1\leq i,j \leq N} |\mathcal{L}_{ij}| \right) <\infty.
\end{align}
Indeed, let $\mathcal{C}$ denote a compact subset of $\mathbb{C}\setminus\{0\}$. Then, $\mathcal{C}$ is closed and bounded, which implies that there exist $s_1,s_2\in\mathcal{C}$ such that $|s_1|\leq|\w|\leq|s_2|$, for all $\w\in\mathcal{C}$. This gives the following bounds
\begin{align}\label{bds compact}
    \log|s_1| \leq \log|\w| \leq \log|s_2| \ \ \ \ \mathrm{ and } \ \ \ \ |s_1|^2 \log|s_1| \leq |\w|^2 \log|\w| \leq |s_2|^2 \log|s_2|,
\end{align}
and so, from Definition~\ref{MN}, we get
\begin{align}\label{N bound}
    \sup_{\w\in\mathcal{C}}\Big|\langle N^{\d k_0}_{D_j D_i} \hat{1}_{D_j}, \hat{1}_{D_i} \rangle\Big| < \infty.
\end{align}
for all $i,j=1,\dots,N,$ with $i\ne j$. Also, using (\ref{B_i}) and (\ref{bds compact}), we get that there exist $\mathcal{F}_1,\mathcal{F}_2\in[0,\infty)$ such that
\begin{align*}
    \mathcal{F}_1 \leq |\mathcal{B}(\w,\d)| \leq \mathcal{F}_2
\end{align*}
for all $\w\in\mathcal{C}$, which gives
\begin{align}\label{B bound}
    \sup_{\w\in\mathcal{C}}\Big|\mathcal{B}(\w,\d)\Big| < \infty.
\end{align}
Applying (\ref{N bound}) and (\ref{B bound}) to the definition of $\mathcal{L}$ in (\ref{2dscat6}), we obtain the desired bound (\ref{A2}).

\subsubsection{Third condition: Approximation convergence}

For the next part, we have to show a convergence result as $\r\rightarrow0$ on the matrix formulations of the problem before and after using (\ref{approx formula}). We will provide this in the setting of three resonators, since the calculations are lengthy and similar for $N\in\mathbb{N}$ particles and so can be easily extrapolated. In this case, we have $\mathcal{L} p_{\r}u = \Big( (\mathcal{L} p_{\r}u)_1, (\mathcal{L} p_{\r}u)_2, (\mathcal{L} p_{\r}u)_3 \Big)^{\top}$, where
\begin{align*}
    (\mathcal{L} p_{\r}u)_i =
    &\langle N_{D_i D_{i+1\lfloor 3 \rfloor}}^{\d k_0} \hat{1}_{D_i},\hat{1}_{D_{i+1\lfloor 3 \rfloor}} \rangle \langle u, \hat{1}_{D_i} \rangle - \mathcal{B}(\w,\d) \langle N_{D_i D_{i+1\lfloor 3 \rfloor}}^{\d k_0} \hat{1}_{D_i},\hat{1}_{D_{i+1\lfloor 3 \rfloor}} \rangle^2 \langle u, \hat{1}_{D_{i+1\lfloor 3 \rfloor}} \rangle\\
    &- \mathcal{B}(\w,\d) \langle N_{D_i D_{i+1\lfloor 3 \rfloor}}^{\d k_0} \hat{1}_{D_i},\hat{1}_{D_{i+1\lfloor 3 \rfloor}} \rangle \langle N_{D_{i+2\lfloor 3 \rfloor} D_i}^{\d k_0} \hat{1}_{D_{i+2\lfloor 3 \rfloor}},\hat{1}_{D_i} \rangle\langle u, \hat{1}_{D_{i+2\lfloor 3 \rfloor}} \rangle,
\end{align*}
for $i=1,2,3$ and we define $\mathcal{F}u$ to be our system before the approximation, \emph{i.e.}
\begin{align*}
    \mathcal{F}u = 
    \begin{pmatrix}
        \langle N_{D_1 D_2}^{\d k_0} u_1,\hat{1}_{D_2} \rangle - \mathcal{B}(\w,\d) \Big[ \langle N_{D_2 D_1}^{\d k_0} u_2,\hat{1}_{D_1} \rangle + \langle N_{D_3 D_1}^{\d k_0} u_3,\hat{1}_{D_1} \rangle \Big] \langle N_{D_1 D_2}^{\d k_0} \hat{1}_{D_1},\hat{1}_{D_2} \rangle \\
        \langle N_{D_1 D_2}^{\d k_0} u_1,\hat{1}_{D_2} \rangle - \mathcal{B}(\w,\d) \Big[ \langle N_{D_2 D_1}^{\d k_0} u_2,\hat{1}_{D_1} \rangle + \langle N_{D_3 D_1}^{\d k_0} u_3,\hat{1}_{D_1} \rangle \Big] \langle N_{D_1 D_2}^{\d k_0} \hat{1}_{D_1},\hat{1}_{D_2} \rangle \\
        \langle N_{D_3 D_1}^{\d k_0} u_3,\hat{1}_{D_1} \rangle - \mathcal{B}(\w,\d) \Big[ \langle N_{D_1 D_3}^{\d k_0} u_1,\hat{1}_{D_3} \rangle + \langle N_{D_2 D_3}^{\d k_0} u_2,\hat{1}_{D_3} \rangle \Big] \langle N_{D_3 D_1}^{\d k_0} \hat{1}_{D_3},\hat{1}_{D_1} \rangle 
    \end{pmatrix}.
\end{align*}
We want to show that
\begin{align} \label{lim approx}
    \lim_{\r\rightarrow0}\| \mathcal{L}p_{\r}u - \mathcal{F}u \|_{\sup} = 0.
\end{align}
Indeed, let us treat this difference at each entry separately. Since, the operators repeat themselves with different indices, and the particles are identical, whatever we show for the first entry holds for the rest. Hence, our study focuses on
\begin{align*}
    \mathcal{W} := &\lim_{\r\rightarrow0} \Big|\langle N_{D_1 D_2}^{\d k_0} \hat{1}_{D_1},\hat{1}_{D_2} \rangle  \langle u, \hat{1}_{D_1} \rangle - \mathcal{B}(\w,\d) \langle N_{D_1 D_2}^{\d k_0} \hat{1}_{D_1},\hat{1}_{D_2} \rangle^2  \langle u, \hat{1}_{D_2} \rangle \\
    &- \mathcal{B}(\w,\d) \langle N_{D_1 D_2}^{\d k_0} \hat{1}_{D_1},\hat{1}_{D_2} \langle N_{D_3 D_1}^{\d k_0} \hat{1}_{D_3},\hat{1}_{D_1} \rangle  \langle u, \hat{1}_{D_3} \rangle - \Big( \langle N_{D_1 D_2}^{\d k_0} u_1,\hat{1}_{D_2} \rangle \\
    &- \mathcal{B}(\w,\d) \Big[ \langle N_{D_2 D_1}^{\d k_0} u_2,\hat{1}_{D_1} \rangle + \langle N_{D_3 D_1}^{\d k_0} u_3,\hat{1}_{D_1} \rangle \Big] \langle N_{D_1 D_2}^{\d k_0} \hat{1}_{D_1},\hat{1}_{D_2} \rangle \Big)\Big|
\end{align*}
We are going to split $\mathcal{W}$ into three differences
$$
\mathcal{W}_1 := \langle N_{D_1 D_2}^{\d k_0} \hat{1}_{D_1},\hat{1}_{D_2} \rangle  \langle u, \hat{1}_{D_1} \rangle - \langle N_{D_1 D_2}^{\d k_0} u_1,\hat{1}_{D_2} \rangle, 
$$
$$
\mathcal{W}_2 := \mathcal{B}(\w,\d) \langle N_{D_2 D_1}^{\d k_0} u_2,\hat{1}_{D_1} \rangle \langle N_{D_1 D_2}^{\d k_0} \hat{1}_{D_1},\hat{1}_{D_2} \rangle - \mathcal{B}(\w,\d) \langle N_{D_1 D_2}^{\d k_0} \hat{1}_{D_1},\hat{1}_{D_2} \rangle^2  \langle u, \hat{1}_{D_2} \rangle
$$
and
$$
\mathcal{W}_3 := \mathcal{B}(\w,\d) \langle N_{D_3 D_1}^{\d k_0} u_3,\hat{1}_{D_1} \rangle \langle N_{D_1 D_2}^{\d k_0} \hat{1}_{D_1},\hat{1}_{D_2} \rangle - \mathcal{B}(\w,\d) \langle N_{D_1 D_2}^{\d k_0} \hat{1}_{D_1},\hat{1}_{D_2} \langle N_{D_3 D_1}^{\d k_0} \hat{1}_{D_3},\hat{1}_{D_1} \rangle  \langle u, \hat{1}_{D_3} \rangle.
$$
We will study them separately to show the convergence result. Let us recall that for $u\in L^2(D)$
\begin{align*}
    N_{D_i D_j}^{\d k_0} u = \hat{K}^{\d k_0}_{D_i D_j}u + R^{(0)}_{D_i D_j} u + (\d k_0)^2 \log(\d k_0 \hat{\g}) R^{(1)}_{D_i D_j} u.
\end{align*}
Then, we have
\begin{align*}
    \mathcal{W}_1 &=  \frac{1}{|D_1|\sqrt{|D_2|}}  \langle N_{D_1 D_2}^{\d k_0} 1_{D_1},1_{D_2} \rangle  \langle u, 1_{D_1} \rangle - \frac{1}{\sqrt{|D_2|}}\langle N_{D_1 D_2}^{\d k_0} u_1,1_{D_2} \rangle \\
    &=  \frac{1}{|D_1|\sqrt{|D_2|}} \langle \hat{K}^{\d k_0}_{D_1 D_2} 1_{D_1}, 1_{D_2} \rangle \langle u,1_{D_1} \rangle - \frac{1}{\sqrt{|D_2|}}\langle \hat{K}_{D_1 D_2}^{\d k_0} u_1,1_{D_2} \rangle + \\
    & \ \ \ + \frac{1}{|D_1|\sqrt{|D_2|}} \langle R^{(0)}_{D_1 D_2} 1_{D_1}, 1_{D_2} \rangle \langle u,1_{D_1} \rangle - \frac{1}{\sqrt{|D_2|}}\langle R_{D_1 D_2}^{(0)} u_1,1_{D_2} \rangle + \\
    & \ \ \ + (\d k_0)^2 \log(\d k_0 \hat{\g}) \left( \frac{1}{|D_1|\sqrt{|D_2|}} \langle R^{(1)}_{D_1 D_2} 1_{D_1}, 1_{D_2} \rangle \langle u,1_{D_1} \rangle - \frac{1}{\sqrt{|D_2|}}\langle R_{D_1 D_2}^{(1)} u_1,1_{D_2} \rangle \right)
\end{align*}
We observe that
\begin{align*}
    &\frac{1}{|D_1|\sqrt{|D_2|}} \langle \hat{K}^{\d k_0}_{D_1 D_2} 1_{D_1}, 1_{D_2} \rangle \langle u,1_{D_1} \rangle - \frac{1}{\sqrt{|D_2|}}\langle \hat{K}_{D_1 D_2}^{\d k_0} u_1,1_{D_2} \rangle = \\
    & = - \frac{1}{|D_1|\sqrt{|D_2|}} \frac{1}{2\pi} \log(\d k_0 \hat{\g}) \int_{D_2} \int_{D_1} \upd y \upd x \int_{D_1} u(x) \upd x - \frac{1}{\sqrt{|D_2|}} \frac{1}{2\pi} \log(\d k_0 \hat{\g}) \int_{D_2} \int_{D_1} u(y) \upd y \upd x \\
    & = - \frac{1}{2\pi} \log(\d k_0 \hat{\g}) \int_{D_1} u(x) \upd x \left[ \frac{1}{|D_1|\sqrt{|D_2|}} |D_1| |D_2| - \frac{1}{\sqrt{|D_2|}} |D_2| \right]\\
    & = 0.
\end{align*}
Also,
\begin{align*}
    &\frac{1}{|D_1|\sqrt{|D_2|}} \langle R^{(0)}_{D_1 D_2} 1_{D_1}, 1_{D_2} \rangle \langle u,1_{D_1} \rangle - \frac{1}{\sqrt{|D_2|}}\langle R_{D_1 D_2}^{(0)} u_1,1_{D_2} \rangle = \\
    & = \frac{1}{|D_1|\sqrt{|D_2|}} \int_{D_2} \int_{D_1} \log|x-y|\upd y \upd x \int_{D_1}u(x)\upd x - \frac{1}{\sqrt{|D_2|}} \int_{D_2} \int_{D_1} \log|x-y| u(y) \upd y \upd x\\
    &= \mathcal{W}_1^{(1)}.
\end{align*}
We know that for $y\in D_1$ and $x\in D_2$, it holds
\begin{align}\label{log bound}
    \log|\a_1 - 2\r| \leq \log |x-y| \leq \log|\a_1 + 2\r|.
\end{align}
This gives
\begin{align} \label{W11 left}
    \sqrt{|D_2|} \log|\a_1 - 2\r| \int_{D_1}u(x)\upd x - \sqrt{|D_2|} \log|\a_1 + 2\r| \int_{D_1} u(x) \upd x \leq \mathcal{W}_1^{(1)}
\end{align}    
and
\begin{align} \label{W11 right}    
    \mathcal{W}_1^{(1)} \leq \sqrt{|D_2|} \log|\a_1 + 2\r| \int_{D_1}u(x)\upd x - \sqrt{|D_2|} \log|\a_1 - 2\r| \int_{D_1} u(x) \upd x.
\end{align}
It is direct that as $\r\rightarrow0$, the left hand side of (\ref{W11 left}) and the right hand side of (\ref{W11 right}), both converge to 0. Thus, as $\r\rightarrow0$,
$$
\mathcal{W}_1^{(1)}\rightarrow0.
$$
Then, 
\begin{align*}
    &\frac{1}{|D_1|\sqrt{|D_2|}} \langle R^{(1)}_{D_1 D_2} 1_{D_1}, 1_{D_2} \rangle \langle u,1_{D_1} \rangle - \frac{1}{\sqrt{|D_2|}}\langle R_{D_1 D_2}^{(1)} u_1,1_{D_2} \rangle = \\
    & = \frac{1}{|D_1|\sqrt{|D_2|}} \int_{D_2} \int_{D_1} \frac{1}{|x-y|} \upd y \upd x \int_{D_1}u(x)\upd x - \frac{1}{\sqrt{|D_2|}} \int_{D_2} \int_{D_1} \frac{ u(y) }{|x-y|} \upd y \upd x\\
    &= \mathcal{W}_1^{(2)}.
\end{align*}
We know that for $y\in D_1$ and $x\in D_2$, it holds
\begin{align}\label{1/x bound}
    \frac{1}{|\a_1 + 2\r|} \leq \frac{1}{|x-y|} \leq \frac{1}{|\a_1 - 2\r|}.
\end{align}
This gives
\begin{align}\label{W12 left}
    \frac{\sqrt{|D_2|}}{|\a_2 + 2\r|} \int_{D_1} u(x) \upd x - \frac{\sqrt{|D_2|}}{|\a_2 - 2\r|} \int_{D_1} u(x) \upd x
    \leq \mathcal{W}_1^{(2)}
\end{align}
and
\begin{align}\label{W12 right}
    \mathcal{W}_1^{(2)} \leq \frac{\sqrt{|D_2|}}{|\a_2 - 2\r|} \int_{D_1} u(x) \upd x - \frac{\sqrt{|D_2|}}{|\a_2 + 2\r|} \int_{D_1} u(x) \upd x
\end{align}
Again, we see that as $\r\rightarrow0$, the left hand side of (\ref{W12 left}) and the right hand side of (\ref{W12 right}), both converge to 0. Thus, as $\r\rightarrow0$,
$$
\mathcal{W}_1^{(2)}\rightarrow0.
$$
Thus, gathering these results, we get that
\begin{align*}
    \mathcal{W}_1 = \mathcal{W}_1^{(1)} + (\d k_0)^2 \log(\d k_0 \hat{\g}) \mathcal{W}_1^{(2)},
\end{align*}
which, at hand, shows that, as $\r\rightarrow0$,
\begin{align*}
    \mathcal{W}_1 \rightarrow 0.
\end{align*}
Let us now show the convergence of $\mathcal{W}_2$ as $\r\to0$. Then, we note that this also gives the convergence of $\mathcal{W}_3$, since the calculations are of the same order. Keeping in mind that $\lim_{\r\rightarrow0}\mathcal{B}(\w,\d)$ is finite, 
we will study
\begin{align*}
    \tilde{\mathcal{W}}_2 := \frac{\mathcal{W}_2}{\mathcal{B}(\w,\d)} = \langle N^{\d k_0}_{D_2 D_1} u_2, \hat{1}_{D_1}\rangle \langle N_{D_1 D_2}^{\d k_0} \hat{1}_{D_1},\hat{1}_{D_2} \rangle - \langle N_{D_1 D_2}^{\d k_0} \hat{1}_{D_1},\hat{1}_{D_2} \rangle^2  \langle u, \hat{1}_{D_2} \rangle,
\end{align*}
which is
\begin{align*}
    \tilde{\mathcal{W}}_2 &= \Big[\langle \hat{K}^{\d k_0}_{D_2 D_1} u_2, \hat{1}_{D_1} \rangle \langle \hat{K}_{D_1 D_2}^{\d k_0} \hat{1}_{D_1},\hat{1}_{D_2} \rangle - \langle \hat{K}_{D_1 D_2}^{\d k_0} \hat{1}_{D_1},\hat{1}_{D_2} \rangle^2 \langle u, \hat{1}_{D_2} \rangle \Big]  + \\
    & \ \ + \Big[\langle R^{(0)}_{D_2 D_1} u_2, \hat{1}_{D_1} \rangle \langle R_{D_1 D_2}^{(0)} \hat{1}_{D_1},\hat{1}_{D_2} \rangle - \langle R_{D_1 D_2}^{(0)} \hat{1}_{D_1},\hat{1}_{D_2} \rangle^2 \langle u, \hat{1}_{D_2} \rangle\Big] +\\
    & \ \ + \Big((\d k_0)^2 \log(\d k_0 \hat{\g})\Big)^2 \Big[ \langle R^{(1)}_{D_2 D_1} u_2, \hat{1}_{D_1} \rangle \langle R_{D_1 D_2}^{(1)} \hat{1}_{D_1},\hat{1}_{D_2} \rangle - \langle R_{D_1 D_2}^{(1)} \hat{1}_{D_1},\hat{1}_{D_2} \rangle^2 \langle u, \hat{1}_{D_2} \rangle \Big] + \\
    & \ \ + \Big[ \langle \hat{K}_{D_1 D_2}^{\d k_0} \hat{1}_{D_1},\hat{1}_{D_2} \rangle \langle R^{(0)}_{D_2 D_1} u_2, \hat{1}_{D_1} \rangle + \langle \hat{K}_{D_1 D_2}^{\d k_0} u_2,\hat{1}_{D_1} \rangle \langle R^{(0)}_{D_2 D_1} \hat{1}_{D_2}, \hat{1}_{D_1} \rangle \\
    & \ \ \ \ \ \ \ \ \ \ \ \ \ \ \ \ \ \ \ \ \ \ \ \ \ \ \ \ \ \ \ \ \ \ \ \ \ \ \ \ \ \ \ \ \ \ \ \ \ \ - 2 \langle \hat{K}_{D_1 D_2}^{\d k_0} \hat{1}_{D_1},\hat{1}_{D_2} \rangle \langle R_{D_1 D_2}^{(0)} \hat{1}_{D_1},\hat{1}_{D_2} \rangle \langle u, \hat{1}_{D_2} \rangle \Big] + \\
    &\ \ + (\d k_0)^2 \log(\d k_0 \hat{\g}) \Big[ \langle \hat{K}_{D_1 D_2}^{\d k_0} \hat{1}_{D_1},\hat{1}_{D_2} \rangle \langle R^{(1)}_{D_2 D_1} u_2, \hat{1}_{D_1} \rangle + \langle \hat{K}_{D_1 D_2}^{\d k_0} u_2,\hat{1}_{D_2} \rangle \langle R^{(1)}_{D_2 D_1} \hat{1}_{D_2}, \hat{1}_{D_1} \rangle \\
    & \ \ \ \ \ \ \ \ \ \ \ \ \ \ \ \ \ \ \ \ \ \ \ \ \ \ \ \ \ \ \ \ \ \ \ \ \ \ \ \ \ \ \ \ \ \ \ \ \ \ - 2 \langle \hat{K}_{D_1 D_2}^{\d k_0} \hat{1}_{D_1},\hat{1}_{D_2} \rangle \langle R_{D_1 D_2}^{(1)} \hat{1}_{D_1},\hat{1}_{D_2} \rangle \langle u, \hat{1}_{D_2} \rangle \Big] + \\
    &\ \ + (\d k_0)^2 \log(\d k_0 \hat{\g}) \Big[ \langle R_{D_1 D_2}^{(0)} \hat{1}_{D_1},\hat{1}_{D_2} \rangle \langle R^{(1)}_{D_2 D_1} u_2, \hat{1}_{D_1} \rangle + \langle R_{D_1 D_2}^{(0)} u_2,\hat{1}_{D_2} \rangle \langle R^{(1)}_{D_2 D_1} \hat{1}_{D_2}, \hat{1}_{D_1} \rangle \\
    & \ \ \ \ \ \ \ \ \ \ \ \ \ \ \ \ \ \ \ \ \ \ \ \ \ \ \ \ \ \ \ \ \ \ \ \ \ \ \ \ \ \ \ \ \ \ \ \ \ \ - 2 \langle R_{D_1 D_2}^{(0)} \hat{1}_{D_1},\hat{1}_{D_2} \rangle \langle R_{D_1 D_2}^{(1)} \hat{1}_{D_1},\hat{1}_{D_2} \rangle \langle u, \hat{1}_{D_2} \rangle \Big] \\
    &=: \tilde{\mathcal{W}}_2^{(1)} + \tilde{\mathcal{W}}_2^{(2)} + \Big((\d k_0)^2 \log(\d k_0 \hat{\g})\Big)^2 \tilde{\mathcal{W}}_2^{(3)} + \tilde{\mathcal{W}}_2^{(4)} \\
    & \ \ \ \ \ \ \ \ \ \ \ \ \ \ \ \ \ \ \ \ \ \ \ \ \ \ \ \ \ \ \ \ \ \ \ \ \ \ \ \ \ \ \ \ \ \ \ \ \ \ + (\d k_0)^2 \log(\d k_0 \hat{\g}) \tilde{\mathcal{W}}_2^{(5)} + (\d k_0)^2 \log(\d k_0 \hat{\g}) \tilde{\mathcal{W}}_2^{(6)}.
\end{align*}
We will consider each of the $\tilde{\mathcal{W}}_2^{(i)}, i = 1,\dots,6$ separately. We observe that
\begin{align*}
    \tilde{\mathcal{W}}_2^{(1)} &= \frac{1}{4\pi^2}\log(\d k_0 \hat{\g})^2 \Big[ \frac{1}{|D_1|\sqrt{|D_2|}} \int_{D_1} \int_{D_2} u(y) \upd y \upd x \int_{D_1} \int_{D_2} \upd y \upd x - \\
    & \quad \quad \quad \quad \quad \quad \quad \quad \quad \quad \quad \quad - \frac{1}{|D_1||D_2|\sqrt{|D_2|}} \left(\int_{D_1} \int_{D_2} \upd y \upd x \right)^2 \int_{D_2} u(y) \upd y \Big] = 0.
\end{align*}
Then,
\begin{align*}
    \tilde{\mathcal{W}}_2^{(2)} &= \langle R_{D_1 D_2}^{(0)} \hat{1}_{D_1},\hat{1}_{D_2} \rangle \Big[ \langle R^{(0)}_{D_2 D_1} u_2, \hat{1}_{D_1} \rangle - \langle R_{D_1 D_2}^{(0)} \hat{1}_{D_1},\hat{1}_{D_2} \rangle \langle u, \hat{1}_{D_2} \rangle \Big].
\end{align*}
We know that, as $\r\to0$, 
$$
\langle R_{D_1 D_2}^{(0)} \hat{1}_{D_1},\hat{1}_{D_2} \rangle \to 0
$$
and, up to changing the indices, from (\ref{W11 left}) and (\ref{W11 right}), we have shown that as $\r\to0$
$$
\langle R^{(0)}_{D_2 D_1} u_2, \hat{1}_{D_1} \rangle - \langle R_{D_1 D_2}^{(0)} \hat{1}_{D_1},\hat{1}_{D_2} \rangle \langle u, \hat{1}_{D_2} \rangle \to 0.
$$
Thus, as $\r\to0$, it holds that
\begin{align*}
    \tilde{\mathcal{W}}_2^{(2)} \to 0.
\end{align*}
In the same reasoning, we have,
\begin{align*}
    \tilde{\mathcal{W}}_2^{(3)} = \Big((\d k_0)\log(\d k_0 \hat{\g})\Big)^2 \langle R^{(1)}_{D_2 D_1} \hat{1}_{D_2}, \hat{1}_{D_1} \rangle \Big[ \langle R^{(1)}_{D_2 D_1} u_2, \hat{1}_{D_1} \rangle - \langle R^{(1)}_{D_2 D_1} \hat{1}_{D_2}, \hat{1}_{D_1} \rangle \Big]
\end{align*}
where, as $\r\to0$,
$$
\langle R_{D_1 D_2}^{(1)} \hat{1}_{D_1},\hat{1}_{D_2} \rangle \to 0
$$
and, up to changing the indices, from (\ref{W12 left}) and (\ref{W12 right}), we have, as $\r\to0$
$$
\langle R^{(1)}_{D_2 D_1} u_2, \hat{1}_{D_1} \rangle - \langle R_{D_1 D_2}^{(1)} \hat{1}_{D_1},\hat{1}_{D_2} \rangle \langle u, \hat{1}_{D_2} \rangle \to 0.
$$
Hence, as $\r\to 0$,
\begin{align*}
    \tilde{\mathcal{W}}^{(3)}_2 \to 0.
\end{align*}
Now,
\begin{align*}
    \tilde{\mathcal{W}}^{(4)}_2 &= \frac{1}{|D_1|\sqrt{|D_2|}} \int_{D_1} \int_{D_2} u(y) \upd y \upd x \int_{D_1} \int_{D_2} \log|x-y| \upd y \upd x +\\
    & \ \ + \frac{1}{|D_1|\sqrt{|D_2|}} \int_{D_1} \int_{D_2} \upd y \upd x \int_{D_1} \int_{D_2} \log|x-y| u(y) \upd y \upd x - \\
    & \ \ - \frac{2}{|D_1||D_2|\sqrt{|D_2|}} \int_{D_1} \int_{D_2} \upd y \upd x \int_{D_1} \int_{D_2} \log|x-y| \upd y \upd x \int_{D_2} u(y) \upd y\\
    &= \sqrt{|D_2|} \int_{D_1} \int_{D_2} \log|x-y| u(y) \upd y \upd x - \frac{1}{\sqrt{|D_2|}} \int_{D_1} \int_{D_2} \log|x-y| \upd y \upd x \int_{D_2} u(y) \upd y.
\end{align*}
Using the bounds (\ref{log bound}), we have
\begin{align*}
    |D_1| \sqrt{|D_2|} \Big(\log|\a_1 - 2\r| - \log|\a_1 + 2\r|\Big) \int_{D_2} u(y) \upd y \leq \tilde{\mathcal{W}}^{(4)}_2,
\end{align*}
and
\begin{align*}
    \tilde{\mathcal{W}}^{(4)}_2 \leq  |D_1| \sqrt{|D_2|} \Big(\log|\a_1 + 2\r| - \log|\a_1 - 2\r|\Big) \int_{D_2} u(y) \upd y,
\end{align*}
which gives, as $\r\to0$,
\begin{align*}
    \tilde{\mathcal{W}}^{(4)}_2 \to 0.
\end{align*}
Then,
\begin{align*}
    \tilde{\mathcal{W}}^{(5)}_2 &= \frac{1}{|D_1|\sqrt{|D_2|}} \int_{D_1} \int_{D_2} u(y) \upd y \upd x \int_{D_1} \int_{D_2} \frac{1}{|x-y|} \upd y \upd x +\\
    & \ \ + \frac{1}{|D_1|\sqrt{|D_2|}} \int_{D_1} \int_{D_2} \upd y \upd x \int_{D_1} \int_{D_2} \frac{1}{|x-y|} u(y) \upd y \upd x - \\
    & \ \ - \frac{2}{|D_1||D_2|\sqrt{|D_2|}} \int_{D_1} \int_{D_2} \upd y \upd x \int_{D_1} \int_{D_2} \frac{1}{|x-y|} \upd y \upd x \int_{D_2} u(y) \upd y\\
    &= \sqrt{|D_2|} \int_{D_1} \int_{D_2} \frac{1}{|x-y|} u(y) \upd y \upd x - \frac{1}{\sqrt{|D_2|}} \int_{D_1} \int_{D_2} \frac{1}{|x-y|} \upd y \upd x \int_{D_2} u(y) \upd y.
\end{align*}
Using the bounds (\ref{1/x bound}), we have
\begin{align*}
    |D_1|\sqrt{|D_2|} \left(\frac{1}{\a_1+2\r}-\frac{1}{\a_1-2\r}\right) \int_{D_2} u(y) \upd y \leq \tilde{\mathcal{W}}^{(5)}_2,
\end{align*}
and 
\begin{align*}
    \tilde{\mathcal{W}}^{(5)}_2 \leq |D_1|\sqrt{|D_2|} \left(\frac{1}{\a_1-2\r}-\frac{1}{\a_1+2\r}\right) \int_{D_2} u(y) \upd y,
\end{align*}
which gives, as $\r\to0$,
\begin{align*}
    \tilde{\mathcal{W}}^{(5)}_2 \to 0.
\end{align*}
Finally,
\begin{align*}
    \tilde{\mathcal{W}}^{(6)}_2 &= \frac{1}{|D_1|\sqrt{|D_2|}} \int_{D_1} \int_{D_2} \log|x-y| u(y) \upd y \upd x \int_{D_1} \int_{D_2} \frac{1}{|x-y|} \upd y \upd x +\\
    & \ \ + \frac{1}{|D_1|\sqrt{|D_2|}} \int_{D_1} \int_{D_2} \log|x-y| \upd y \upd x \int_{D_1} \int_{D_2} \frac{1}{|x-y|} u(y) \upd y \upd x - \\
    & \ \ - \frac{2}{|D_1||D_2|\sqrt{|D_2|}} \int_{D_1} \int_{D_2} \log|x-y| \upd y \upd x \int_{D_1} \int_{D_2} \frac{1}{|x-y|} \upd y \upd x \int_{D_2} u(y) \upd y.
\end{align*}
Here, we combine the bounds (\ref{log bound}) and (\ref{1/x bound}) and get
\begin{align*}
   2|D_1|\sqrt{|D_2|} \left( \frac{\log|\a_1-2\r|}{|\a_1 + 2\r|} - \frac{\log|\a_1+2\r|}{|\a_1 - 2\r|} \right) \int_{D_2} u(y) \upd y \leq \tilde{\mathcal{W}}^{(6)}_2,
\end{align*}
and
\begin{align*}
    \tilde{\mathcal{W}}^{(6)}_2 \leq 2|D_1|\sqrt{|D_2|} \left( \frac{\log|\a_1+2\r|}{|\a_1 - 2\r|} - \frac{\log|\a_1-2\r|}{|\a_1 + 2\r|} \right) \int_{D_2} u(y) \upd y,
\end{align*}
which gives, as $\r\to0$,
\begin{align*}
    \tilde{\mathcal{W}}^{(6)}_2 \to 0.
\end{align*}
Thus, we have shown that for all $i=1,\dots,6$, as $\r\to0$,
\begin{align*}
    \tilde{\mathcal{W}}^{(i)}_2 \to 0,
\end{align*}
which shows that
\begin{align*}
    \tilde{\mathcal{W}}_2 \to 0, \ \ \ \ \text{ as } \r\to0.
\end{align*}
Also, repeating these calculation and re-indexing, we get
\begin{align*}
    \mathcal{W}_3 \to 0, \ \ \ \ \text{ as } \r\to0.
\end{align*}
Therefore, we have that 
\begin{align}\label{A4}
    \mathcal{W} = 0,
\end{align}
and hence, (\ref{lim approx}) follows.

Let us now move to the last part of the proof. We observe that for each $i=1,\dots,N,$
\begin{align*}
    \| u_i - p_{\r}u_i \|_{L^2(D)} \leq \| u \|_{L^2(D)} + \| p_{\r} u \|_{L^2(D)} = 2 \| u \|_{L^2(D)},
\end{align*}
where we have used (\ref{(A1)}), and we have that,
\begin{align*}
    \| u_i \|_{L^2(D)} = \left( \int_{D_i} |u(y)|^2 \upd y \right) \to 0, \ \ \ \text{ as } \r\to 0.
\end{align*}
Therefore, we obtain
\begin{align*}
    \| u_i - p_{\r}u_i \|_{L^2(D)} \to 0, \ \ \ \text{ as } \r\to 0,
\end{align*}
and so, for each subsequence of $\r'\in\mathbb{R}$, such that $\r'\to0$,
\begin{align}\label{A3}
    \lim_{\r'\to0} \| u_i - p_{\r'}u_i \|_{L^2(D)} = 0.
\end{align}

Hence, since (\ref{(A1)}), (\ref{A2}) and (\ref{A3}) hold, we have shown that all the assumptions of Theorem 2.10 in \cite{BLRM} hold. Thus, we conclude that, the approximation formula (\ref{approx formula}) holds.
\end{proof}

\bibliographystyle{abbrv}
\bibliography{references}{}

\end{document}